\documentclass[journal]{IEEEtran}
\usepackage{cite}

\ifCLASSINFOpdf
\else
\fi
\usepackage[cmex10]{amsmath}
\usepackage{array}
\usepackage{fixltx2e}
 \usepackage{dblfloatfix}

\ifCLASSOPTIONcaptionsoff
  \usepackage[nomarkers]{endfloat}
 \let\MYoriglatexcaption\caption
 \renewcommand{\caption}[2][\relax]{\MYoriglatexcaption[#2]{#2}}
\fi
\usepackage{url}

\usepackage{subfigure}

\usepackage{verbatim} 
\usepackage{mathtools} 
\usepackage{amssymb} 
\usepackage{color} 
\usepackage{pgfplots}
\usepackage{balance}

\usepgfplotslibrary{groupplots}

\usepackage[english]{babel}
\usepackage{amsmath,amssymb,amsfonts,amsthm}
\usepackage{algorithm}
\usepackage{algpseudocode}
\usepackage{textcomp}
\usepackage{color}
\usepackage{graphicx}

\usepackage{multicol}
\usepackage{dcolumn}
\usepackage{lipsum}

\newtheorem {lem}{Lemma}
\newtheorem {thm}{Theorem}

\newtheorem {rem}{Remark}

\newtheorem {defi}{Definition}
\newtheorem {cor}{Corollary}

\title{$PI^hD^{n-1}$ synchronization of higher-order nonlinear systems with a recursive Lyapunov approach}

\author{ Davide~Liuzza, Dimos~V.~Dimarogonas and Karl~H.~Johansson
\thanks{D. Liuzza, D. V. Dimarogonas and K. H. Johansson are with ACCESS Linnaeus Centre and School of Electrical Engineering, Royal Institute of Technology, Stockholm, Sweden. Emails: {\tt\small \{liuzza, dimos, kallej\}@kth.se}.
}
}

\begin{document}

\newcommand{\R}{\mathbb{R}}

\maketitle

\begin{abstract}
This paper investigates the problem of synchronization for nonlinear systems. Following a Lyapunov approach, we firstly study global synchronization of nonlinear systems in canonical control form with both distributed proportional-derivative and proportional-integral-derivative control actions of any order. 
To do so, we develop a constructive methodology and generate in an iterative way inequality constraints on the coupling matrices which guarantee the solvability of the problem or, in a dual form, provide the nonlinear weights on the coupling links between the agents such that the network synchronizes. The same methodology allows to include a possible distributed integral action of any order to enhance the rejection of heterogeneous disturbances.The considered approach does not require any dynamic cancellation, thus preserving the original nonlinear dynamics of the agents. 
The results are then  extended to linear and nonlinear systems admitting a canonical control transformation.
Numerical simulations validate the theoretical results.
\end{abstract}


%
\IEEEpeerreviewmaketitle

\section{Introduction}

%
%
%
%

\IEEEPARstart{S}{ynchonization} 
of networked systems has been widely studied in the last decade by different research  communities  \cite{ne:03,bola:06,bape:02}. Strategies allowing to reach an agreement among dynamical agents with only local interactions
have found successful applications in mobile robots and unmanned aerial vehicles formation control, distributed sensors communication, platooning and formation control \cite{ar:07,caan:08,ol:07}. Other relevant applications are related to the synchronization of biochemical oscillators and to the control of distributed large-scale systems, with a particular attention to the class of electrical power networks and smart grids \cite{hast:12,hich:06,dobu:12}.

Research on distributed control and synchronization is often focused on finding conditions and control laws able to steer the system to a common synchronous trajectory.

Specifically, starting from the consensus problem for single integrator nodes, the problem of synchronization has been gradually and extensively extended to  linear systems, first with assumptions on the eigenvalues of the dynamical matrix or input matrix \cite{scse:09,sesh:09} and later under the mild assumption on the controllability and detectability alone of the linear systems \cite{lidu:10,zhle:11}. So,  for the class of linear systems, general results are currently available.
Also research on synchronization of nonlinear systems has generated many results. However, due to the intrinsic difficulty, synchronization of nonlinear systems is still under active investigation. 

Nowadays, various methodologies aim at studying synchronization for wide classes of nonlinear systems. Approaches include Lyapunov methods \cite{lihizh:12,lich:06}, contraction analysis \cite{wasl:05,dili:14} and passivity and incremental dissipativity \cite{ar:07,scar:10,lihi:11}

Other authors focus on synchronization of agents whose model appears in canonical control form, also called companion form \cite{slli:91}. This class of results is known as higher-order synchronization and explicitly exploits the structure of the dynamical model. 

{Specifically, Lyapunov methods are considered, among others, in \cite{lihizh:12,lich:06,liji:08,lili:11,yuchca:11,wehu:13,lich:15,dedi:15}. These papers offer a huge spectrum of approaches for the synchronization problem. Without going to much into details, these works explore the possibility to leverage on: bounded Jacobian assumption, linear systems with additional Lipschitz nonlinearity and the existence of the solution of suitable LMIs, hypothesis on inequalities constraints for the nonlinear dynamics, external reference pinner nodes.}

{Specifically, consensus among second-order integrators and higher-order integrators has been addressed \cite{reat:05,yuch_consensus:10,yuzh:13,remo:06,hoga:07,yuchcaku:10,soca:10,re:08,yuch:11,lixi:13}, following different approaches, such as studying the determinant of the overall networked linear system or via ensuring that the polynomial obtained considering the eigenvalue problem on the companion dynamical systems' matrix and the coupling feedback are Hurwitz.}
One of the motivations behind these studies is related to the fact that several dynamical systems, e.g. mechanical systems, are naturally described in canonical control form and, in particular, higher-order integrators are a more realistic model of mobile robotic vehicles than the simple integrators. 

{The papers reviewed above strongly rely on tools for linear systems or on the specific structure of companion form of higher-order integrators and  their extension to nonlinear systems appears to be a non-trivial task.}

{Lyapunov methods for second-order integrators are considered in \cite{hoga:07} and \cite{yuchcaku:10}, in which a Lyapunov function specific for the second-order case is adopted. A specific second-order integrator Lyapunov approach is also considered in \cite{soca:10}, where the presence of an external pinner is also required, while in \cite{re:08} the specific second-order consensus is considered when bounded control actions are required.
The case of higher-order systems with nonlinear dynamics is instead studied in \cite{lixi:13}. In that paper, the specific cases of first-order and second-order nonlinear systems are considered and, for these two cases, two suitable Lyapunov functions are introduced to prove convergence. The extension to higher-order nonlinear dynamics is not addressed in this work.
In general, although these papers allow to consider nonlinear dynamics via a Lyapunov function, the results appear to be specific to the order and the problem considered and, therefore, not straightforward to scale to any arbitrary system's order.}

{In \cite{dale:11}, synchronization of second-order nonlinear dynamics is addressed via a nonlinear compensation through a neural network and the presence of an external reference. This approach is further extended in \cite{zhle:10,zhle:12,bile:14} for higher-order nonlinear systems. Although such results provide a suitable methodology for addressing the higher-order nonlinear synchronization, the methodology is not applicable to the free synchronization problem where the aim is to preserve the original nonlinear dynamics of the agents while studying an emerging common behaviour without permanently forcing the overall system.}

{Motivated by the need for providing a general framework for the free synchronization problem, in this paper we study the higher-order free synchronization for nonlinear systems of any degree considering local state feedback. 
Referring to the previous literature on this problem, we compare our results with the strategies in \cite{reat:05,yuch_consensus:10,yuzh:13,remo:06,hoga:07,yuchcaku:10,soca:10,re:08,yuch:11,lixi:13}. In our case, nonlinear dynamics are allowed and therefore a Lyapunov approach is developed. However, differently from what done in \cite{reat:05,yuch_consensus:10,yuzh:13,remo:06,hoga:07,yuchcaku:10,soca:10,re:08,yuch:11,lixi:13}, we do not focus our investigation on a specific system's order but instead derive results for general degree higher-order systems. Also, compared to \cite{zhle:10,zhle:12,bile:14}, no dynamic cancellation (i.e. reduction to a higher-order consensus) is needed, thus preserving the free system motion.}

{More specifically, we address the problem via finding a Lyapunov function whose structure is based on the system's order considered.} Therefore, called $n$ the order of the nonlinear agents, a Lyapunov function is derived via a suitable algorithm that generates, up to iteration $n$, a set of appropriate matrices. These matrices, blocked together in a specific way depending on the order $n$, will constitute the core of the Lyapunov function expression, which in turn will prove free synchronization. A key novelty of the approach followed in this paper, with respect to the literature, is that the conducted  analysis is constructive, providing in an iterative way inequality constraints on the coupling matrices which guarantee the solvability of the problem or, in a dual form, providing the nonlinear weights on the coupling links between the agents such that the network synchronizes. 
The given procedure relies on the iterative computation of the solution of a system of three second-order inequalities  that for this reason are, contrary to other approaches in the literature {(see for example \cite{yuch:11} for the case of networked integrators)}, computable in an easier way.

Also, we believe that the  analysis/synthesis method via a constructive Lyapunov function represents a relevant theoretical achievement due to its generality and scalability.  
Furthermore, the approach naturally encompasses the possibility to have distributed integral control actions of any order, i.e.,  distributed $PI^hD^{n-1}$ controllers, with $h\geq 0$ being the degree of the integral action, without any additional hypothesis. Such integral action can be used to attenuate possible distributed and heterogeneous disturbances acting on the interconnected plants. As shown in \cite{frya:06}, an integral action significantly enhances the  performances of the closed loop system. 

{We note here that, the resulting distributed $PI^hD^{n-1}$ controllers have an analogous structure to the ones in \cite{chwa:16,wach:15}. These latter papers address the flocking problem of a team of mobile robots following a polynomial reference trajectory. Such mobile agents are modelled with single \cite{wach:15} and  higher-order \cite{chwa:16} integrators and $PI^n$ and $PI^{l_m-m}D^{m-1}$ containment controllers are, respectively, designed. To prove convergence, the adopted methodology exploits a pole-placement technique for the individual linear system and then solves a Lyapunov equation on the overall linear systems.
In \cite{chwa:16}, a discrete time version of the proposed strategies is also developed.
Despite the analogy of the controllers' structure, however, these works differ from the results presented here in the control goal, the agents' model and the analytical techniques adopted. 
}

As a further contribution of our paper, the approach studied for higher-order nonlinear systems is extended to the relevant class of interconnected nonlinear systems admitting a canonical control transformation, resulting in a distributed nonlinear control action which guarantees the synchronization of the network.
Classes of problem studied in the literature, such as second-order and higher-order consensus can be seen as special cases of such general framework.
The particular case of linear systems is also addressed as a corollary of such general framework, thus resulting in the sufficient condition of controllability of the linear systems, as already showed in a different way in \cite{lidu:10}. However, it is worth noticing that also for the case of linear systems, the approach presented in the paper naturally allows to explicitly consider integral control actions of any order for possible disturbances rejections.

The paper is organized in the following way. A mathematical background and the problem statement can be found in Section \ref{sec:mathematical_background} and Section \ref{sec:problem_formulation}, respectively. In Section \ref{sec:ditributed_gain_selection}  the aforementioned iterative algorithms are presented. The synchronization of systems in companion form is proved in Section \ref{sec:synch_canonical_form} both for $PD^{n-1}$ and $PI^hD^{n-1}$ local control laws, while an extension to controllable systems is addressed in Section \ref{sec:synch_canonical_transformation}. Numerical examples are illustrated in Section \ref{sec:numerical_examples}, while concluding remarks and future work are given in Section \ref{sec:conlusion_future_work}.


\section{Mathematical background}\label{sec:mathematical_background}

\subsection{Matrix Analysis}
Here we report some concepts of matrix analysis that will be useful in the rest of the paper  \cite{hojo:87}. 

{Let us consider} a generic square matrix  $A\in \Bbb{R}^{n\times n}$. For any index $k \in \{1,\dots,n\}$, the $k\times k$ top left submatrix obtained from $A$, so considering the entries that lie in the first $k$ rows and columns of $A$, is called a {\em leading principal submatrix} and its determinant is called  {\em leading  principal minor}. In analogous way, the $k\times k$ bottom right submatrix is called {\em trailing principal submatrix} and its determinant {\em trailing principal minor}.

Two matrices $A, B\in \Bbb{R}^{n\times n}$ are said to be {\em commutative} if $AB=BA$. Furthermore, they are said to be {\em simultaneously diagonalizable} if there exists a nonsingular matrix $S\in \Bbb{R}^{n\times n}$ such that $S^{-1}AS$ and $S^{-1}BS$ are both diagonal. The following result hold.
\begin{lem}
Let $A,B\in \Bbb{R}^{n\times n}$ be simultaneously diagonalizable. Then they are commutative.
\end{lem}

Let $A\in \Bbb{R}^{n\times n}$ be any symmetric matrix, i.e. $A=A^T$. Then the eigenvalues of $A$ are real and the eigenvectors constitutes an orthonormal basis for $A$. We denote with $\mathrm{eig}(A)$ the set containing the eigenvalues of A and with $\lambda_{\min}(A)=\min_{\lambda_i\in \mathrm{eig}(A)}\lambda_i$ and with $\lambda_{\max}(A)=\max_{\lambda_i\in \mathrm{eig}(A)}\lambda_i$ the minimum and maximum eigenvalue of $A$, respectively. For a symmetric matrix the following results hold.

\begin{lem}\textrm{(Rayleigh)}
\newline
Let $A\in \Bbb{R}^{n\times n}$ be a symmetric matrix. Then, for all $y\in\R^n$  it holds
$
\lambda_{\min}y^Ty\leq y^T A y\leq \lambda_{\max} y^Ty.
$
\end{lem}


\begin{lem}\textrm{ (Sylvester's criterion)}
\newline
Let $A\in \Bbb{R}^{n\times n}$ be a symmetric matrix. Then, $A$ is {positively defined} iff every leading (respectively, trailing) principal minor of A is positive (including the determinant of $A$).
\end{lem}

\subsection{Lie algebra and weak-Lipschitz functions}

Here we give some useful definitions and basic concepts on differential geometry (for more details see also \cite{kh:02,slli:91}) and the definition of weak-Lipschitz functions that will be useful in the rest of the paper.

\begin{defi}
A function $T(x):\R^n\mapsto \R^n$ defined in a region $\Omega\subseteq \R^n$ is said to be a {\em diffeomorphism} if it is smooth and invertible, with inverse function $T^{-1}(x)$ smooth. 
\end{defi}

Given a smooth scalar function $h(x):\R^n\mapsto \R$, its gradient will be denoted by the row vector $\frac{\partial}{\partial x}h(x)=\left[\frac{\partial}{\partial x_1}h(x),\dots,\frac{\partial}{\partial x_n}h(x)\right]$. In the case of vector function $f(x):\R^n\mapsto\R^n$, with the same notation $\frac{\partial}{\partial x}f(x)$ we denote the Jacobian matrix of $f(x)$.
The following definitions can be now given. 

\begin{defi}
Let us consider a smooth scalar function $h(x):\R^n\mapsto \R$ and a smooth vector field $f(x):\R^n\mapsto\R^n$, the {\em Lie derivative of $h$ with respect to $f$} is the scalar function defined as $\mathcal{L}_{f}h(x):=\frac{\partial}{\partial x}h(x)f(x)$. 
\end{defi}
Multiple Lie derivative can be easily written by recursively extending the notation  as $\mathcal{L}^k_{f}h(x)=\mathcal{L}_{f}\left(\mathcal{L}^{k-1}_f h\right)$, for $k=1,2,\dots$, and with $\mathcal{L}^0_{f}h(x)=h$.  

\begin{defi}
Let us consider two smooth vector fields $f(x),g(x):\R^n\mapsto\R^n$, the {\em Lie bracket of $f$ and $g$} is the vector field defined as $\mathrm{ad}_f g(x)= \frac{\partial}{\partial x}g\, f- \frac{\partial}{\partial x}f\, g$. 
\end{defi}
Analogously to what done for the Lie derivative, multiple Lie bracket can be defined as $\mathrm{ad}_f^{k} g=\mathrm{ad}_f \left(\mathrm{ad}_f^{k-1} g \right)$, for $k=1,2,\dots$, and with $\mathrm{ad}_f^{0} g=g$.

\begin{defi}
A set of linearly independent vector fields $\{f_1(x),\dots, f_m(x)\}$ is said to be {\em involutive} if and only if, for all $i,j$, there exist scalar functions $\alpha_{ijk}(x): \R^n\mapsto \R$ such that
$
\mathrm{ad}_{f_i}f_j(x)=\sum_{k=1}^m a_{ijk}(x)f_k(x).
$
\end{defi}  

\begin{defi}
A function $f(t,x): \R^+\times\R^n\mapsto \R^m$ is said to be {\em globally Lipschitz} with respect to $x$ if $\forall x,y \in \R^n,\,\forall t\geq 0$ there exists a constant $w>0$ s.t.
$
\|f(t,x)-f(t,y)\|\leq w \|x-y\|
$.
\end{defi}

\begin{defi}
A function $f(t,x): \R^+\times\R^n\mapsto \R$ is said to be {\em globally weak-Lipschitz} with respect to $x$ if $\forall x,y \in \R^n,\,\forall t\geq 0,\forall i\in\{1,\dots,n\}$ there exists a constant $w>0$ s.t.
$
(x_i-y_i)[f(t,x)-f(t,y)]\leq w \|x-y\|^2
$,
with $x_i$ and $y_i$ being the $i$-th element of vector $x$ and $y$ respectively.
\end{defi}

The following lemma points out a relation between Lipschitz and weak-Lipschitz functions.

\begin{lem}\label{lem:weak_Lipschitz}
A Lipschitz function $f(t,x)= \R^+\times\R^n\mapsto \R$, with Lipschitz constant $w$, is also weak-Lipschitz with the same constant $w$.
\end{lem}

\begin{proof}
Let us introduce the function $F_i(t,x)\in\R^n$ whose $i$-th entry is $f(t,x)$, while the other are null. It is immediate to observe that $\|F_i(t,x)-F_i(t,y)\|=\|f(t,x)-f(t,y)\|$. So, the lemma is proved considering, for all $i\in\{1,\dots,n\}$, the following relation
\begin{align*}
(x_i-y_i)[f(t,x)-f(t,y)]&=(x-y)^T[F_i(t,x)-F_i(t,y)]\\
& \leq w \|x-y\|^2.
\end{align*}
\end{proof}

\begin{rem}
In this paper we will assume that the function $f(t,x^{(i)})$ of the dynamical model given later in \eqref{eq:canonical_control_form} is weak-Lipschitz. However, as also reported in \cite{isma:13}, in presence of synchronization in a compact invariant set, this condition can be replaced by the assumption of locally Lipschitz $f(t,x^{(i)})$. Indeed, each locally Lipschitz function can be extended outside a compact set by appropriate extension theorems.   
\end{rem}

\section{Problem formulation}\label{sec:problem_formulation}
The aim of this paper is to study free synchronization for  multi-agent systems whose dynamics can be expressed in the canonical control form. 

More in detail, a dynamical agent $\dot{x}^{(i)}=X(t,u^{(i)},x^{(i)})$, with $x^{(i)}\in\R^n,u^{(i)}\in\R,t\in[0,+\infty)$ is said to be in {\em canonical control form} or {\em companion form} \cite{slli:91} when it is in the following form
\begin{eqnarray}\label{eq:canonical_control_form}
\dot{x}_1^{(i)}&=&x_2^{(i)} \nonumber \\
&\vdots & \nonumber \\
\dot{x}_n^{(i)}&=&f(t,x^{(i)})+g(t,x^{(i)})u^{(i)}, \nonumber\\
\end{eqnarray}
with $x^{(i)}=\left[x_1^{(i)},\dots, x_n^{(i)}\right]^T$ and with $x^{(i)}(0)=x_0^{(i)}$.  In this paper we will consider the case of\footnote{Notice that when a nonlinear system can be transformed in companion form, this condition is always guaranteed by the transformation procedure itself \cite{slli:91}.} $g(t,x^{(i)}(t))\neq 0$, $\forall t\geq 0$, and so the control input can be rewritten as $u^{(i)}=1/g(t,x^{(i)}(t))\tilde{u}^{(i)}$, with $\tilde{u}^{(i)}\in\R$.

The problem of free synchronization of a multi-agent system is formally defined in what follows.

\begin{defi}
A multi-agent system of identical agents $\dot{x}^{(i)}=X(t,u^{(i)},x^{(i)})$, with $i=1,\dots, N$, is {\em free synchronizable}, if for all the agents there exists a distributed control law $u_i=u_i(t,x_i,x_j)$ with $j\in \mathcal{N}_i$ such that
\begin{subequations}
\begin{align}
&\lim_{t\rightarrow \infty}\|x^{(i)}(t)-x^{(j)}(t)\| =0  &\forall i,j=1,\dots,N, \label{eq:synchronization_difference} \\
&\lim_{t\rightarrow \infty}\|u^{(i)}(t)\| =0  &\forall i=1,\dots,N. \label{eq:synchronization_input}
\end{align}
\end{subequations}
\end{defi}

The goal of this paper is to study the free synchronization of a multi-agent system with agents' dynamics expressed in the companion form \eqref{eq:canonical_control_form} or that can be transformed in such canonical form. We will give conditions  under which the problem of finding a distributed $u^{(i)}$ for each agent able to guarantee conditions \eqref{eq:synchronization_difference}-\eqref{eq:synchronization_input} is solvable. Furthermore, our proofs will be based on a constructive method, so a proportional-derivative ($PD^{n-1}$) and proportional-integral-derivative ($PI^hD^{n-1}$) control law able to synchronize the agents will be explicitly given. 
Specifically, in Section \ref{sec:synch_canonical_form} the problem of synchronization of systems in canonical control form will be addressed, while in Section \ref{sec:synch_canonical_transformation} the results will be extended to the relevant case of systems admitting a canonical transformation. 
Defining the average state trajectory as $\bar{x}(t):=[\bar{x}_1^T(t),\dots, \bar{x}_n^T(t)]^T\in\R^{n}$, with each $\bar{x}_k\in\R$  given by  
$
\bar{x}_k(t)=\frac{1}{N}\sum_{j=1}^N x_k^{(j)}(t)
$
we can define the stack error trajectory as $e:=\left[e_1^T,\dots,e_n^T\right]^T\in\R^{nN}$, with $e_k:=\left[e_k^{(1)},\dots,e_k^{(N)}\right]^T=x_k-\bar{x}_k 1_N$, with  $1_N$ vector of $N$ unitary entries.
It is easy to see that condition \eqref{eq:synchronization_difference} can be equivalently stated in the alternative way
$
\lim_{t\rightarrow \infty}\|e(t)\| =0 
$.

\section{Synchronization couplings constraints}\label{sec:ditributed_gain_selection}
In this section we identify, via an iterative procedure, a class of feedback gain matrices that suffices to achieve free synchronization for systems in companion form. Specifically, instead of using a closed form for identifying the conditions on the feedback gains which guarantee the synchronization, we will define it via such a procedure. The advantage is that, in this way, $PI^hD^{n-1}$ controllers can be defined in a general way and the results can be proven considering any arbitrary degree. 

When the case of a specific communication topology have to be considered, a second iterative procedure is also presented which further imposes on the feedback gains the topology constraint.
As we already said, our main purpose is to investigate the solvability of the higher-order free synchronization problem. However, since the methodology is constructive, the derived conditions 
can also be  used to either check if a given weighted topology allows synchronization or to synthesize distributed gains able to enforce synchronization.

We start giving the following definition.

\begin{defi}\label{def:L_N_matrices}
A symmetric matrix $L\in\R^{N\times N}$ is said to be an $\mathfrak{L}_N$ {\em matrix} if $L1_N=0_N$ and for its eigenvalues $\lambda_1,\dots,\lambda_N$ it holds that
$
0=\lambda_1<\lambda_2\leq\dots\leq \lambda_N
$,
where $1_N$ and $0_N$ are vectors of $N$ unitary and null entries respectively.
Furthermore, we denote with $\mathfrak{L}_N${\em -class}, the set of all $\mathfrak{L}_N$ matrices.
\end{defi}
Notice that the $N\times N$ Laplacian matrices \cite{goro:01} belong to the $\mathfrak{L}_N${\em -class}. However, the $\mathfrak{L}_N${\em -class} is more generic since we do not require the off diagonal elements of the matrix to be non positive and, furthermore, no specific structure of the matrices is a priori assumed.

Given $n, N\in \Bbb{N}$ such that $n, N\geq 2$, let us consider the matrices $\{L_{n-k}\}_{k\in\mathcal{K}}\in\mathfrak{L}_N${\em -class}, with $\mathcal{K}=\{0,\dots,n-1\}$ and pair-wise simultaneously diagonalizable. The orthonormal basis of the $L_{n-k}$ matrices is denoted as 
$
\left\{v^{(1)}, v^{(2)},\dots v^{(N)}\right\}, 
$
with $v^{(1)}=\nu$ {and $\nu=1/N \cdot 1_N$ as stated in Section \ref{sec:problem_formulation}}. For each matrix $L_{n-k}$, we denote with $\lambda_{n-k}^{(i)}$ the eigenvalue corresponding to the eigenvector $v^{(i)}$, for all $i\in\{2,\dots,N\}$, while $\lambda_{n-k}^{(1)}=0$ by Definition \ref{def:L_N_matrices}.
The algorithmic criteria we are going to give aim at identifying a class of synchronizing distributed feedback assigning spectral properties to the matrices $\{L_{n-k}\}_{k\in\mathcal{K}}$ and thus constraining their selection. In particular, for each eigenvalue $\lambda_{n-k}^{(i)}$ associated with eigenvector $v^{(i)}$, with $i\in\mathcal{I}=\{2,\dots,N\}$, we consider inequality constraints via an iterative procedure. 

First, let us consider the initialization $\lambda_0^{(i)}=0$; $0<\lambda_{n-1}^{(i)}<\lambda_{n}^{{(i)}^2}$; $\alpha_{n-1}^{(i)}=\min \mathrm{eig}\{A_{n-1}^{(i)}\}$; $\beta_{n-1}^{(i)}=\lambda_{n}^{{(i)}^2}-\lambda_{n-1}^{(i)}$; $\gamma_{n-1}^{(i)}=1$, with
\begin{equation*}
A_{n-1}^{(i)}= \left[\begin{array}{cc}
2\lambda_{n-1}^{(i)}\lambda_{n}^{(i)} & \lambda_{n-1}^{(i)}\\
\lambda_{n-1}^{(i)} & \lambda_{n}^{(i)}
\end{array}
\right].
\end{equation*} 
It is easy to see that the coefficients $\alpha_{n-1}^{(i)},\beta_{n-1}^{(i)},\gamma_{n-1}^{(i)}$ are strictly positive. Furthermore, for $k=2,\dots,n-1$, we define the iterative terms $\alpha_{n-k}^{(i)}=\min \mathrm{eig}\{A_{n-k}^{(i)}\}$; $\beta_{n-k}^{(i)}=\min \mathrm{eig}\{B_{n-k}^{(i)}\}$; $\gamma_{n-k}^{(i)}=\gamma_{n-k+1}^{(i)}+2\lambda_{n-k+2}^{(i)}$, with
\begin{equation*}
A_{n-k}^{(i)}= \left[\begin{array}{cc}
2\lambda_{n-k}^{(i)}\lambda_{n-k+1}^{(i)} &\gamma_{n-k}^{(i)} \lambda_{n-k}^{(i)}\\
\gamma_{n-k}^{(i)} \lambda_{n-k}^{(i)}& \alpha_{n-k+1}^{(i)}
\end{array}
\right], 
\end{equation*}
\begin{equation*}
B_{n-k}^{(i)}= \left[\begin{array}{cc}
\lambda_{n-k+1}^{{(i)}^2}-2\lambda_{n-k}^{(i)}\lambda_{n-k+2}^{(i)} &-\frac{1}{2}\gamma_{n-k+1}^{(i)} \lambda_{n-k}^{(i)}\\
-\frac{1}{2}\gamma_{n-k+1}^{(i)} \lambda_{n-k}^{(i)} & \beta_{n-k+1}^{(i)}
\end{array}
\right].
\end{equation*}
For convenience we also define $B_0^{(i)}$ and $\beta_0^{(i)}$ by iterating the above $B_{n-k}^{(i)}$ and $\beta_{n-k}^{(i)}$ up to step $k=n$.

Taking into account the above definitions, Algorithm \ref{alg:spectral_constranit_centralized_info} considers for each eigenvector $v^{(i)}$, with $i\in\mathcal{I}$, a particular choice on the corresponding eigenvalues $\lambda_{n-k}^{(i)}$, with $i\in\mathcal{I}$ and $k\in\mathcal{K}$, in order to generate spectral constraints on the matrices $\{L_{n-k}\}_{k\in\mathcal{K}}$. In particular, each $L_{n-k}$ is computed as $L_{n-k}=UD_{n-k}U^T$, with matrices $U=[\nu|v^{(2)}|\dots|v^{(n)}]$ and $D_{n-k}=diag\{0,\lambda_{n-k}^{(2)},\dots,\lambda_{n-k}^{(N)}\}$.

\begin{algorithm}
\caption{Spectral constraints assignment}\label{alg:spectral_constranit_centralized_info}

\begin{algorithmic}[1]

\ForAll{i=2, \dots, N}

\For{k=2,\dots, n-1}

\State Compute $\scriptstyle \alpha_{n-k+1}^{(i)}$  

\State Compute $\scriptstyle \gamma_{n-k}^{(i)}$  

\State Choose a $\scriptstyle \lambda_{n-k}^{(i)}$ that satisfies the following inequalities

\begin{subequations}
\begin{align}
\scriptstyle  & \scriptstyle \lambda_{n-k}^{(i)}>0, \label{eq:system_lambda_n_k_positive} \\
\scriptstyle & \scriptstyle \lambda_{n-k}^{(i)}<\frac{2\lambda_{n-k+1}^{(i)}\alpha_{n-k+1}^{(i)}}{\gamma_{n-k}^{{(i)}^2}}, \label{eq:system_lambda_n_k_first_order_inequality}\\
\scriptstyle &\scriptstyle \gamma_{n-k+1}^{{(i)}^2}\lambda_{n-k}^{{(i)}^2}+8\lambda_{n-k+2}^{(i)}\beta_{n-k+1}^{(i)}\lambda_{n-k}^{{(i)}}-4\lambda_{n-k+1}^{{(i)}^2}\beta_{n-k+1}^{(i)}<0. \label{eq:system_lambda_n_k_second_order_inequality}
\end{align}
\end{subequations}

\State Define $\scriptstyle B_{n-k}^{(i)}$

\State Compute $\scriptstyle \beta_{n-k}^{(i)}$

\EndFor

\EndFor

\For{k=0,\dots, n-1}
\State Set $\scriptstyle D_{n-k}\gets diag\{0,\lambda_{n-k}^{(2)},\dots,\lambda_{n-k}^{(N)}\}$

\State Set $\scriptstyle L_{n-k}\gets UD_{n-k}U^T$

\EndFor

\end{algorithmic}
\end{algorithm}

Notice that, the inequalities \eqref{eq:system_lambda_n_k_positive}\--\eqref{eq:system_lambda_n_k_second_order_inequality} are always feasible, since the right hand side of \eqref{eq:system_lambda_n_k_first_order_inequality} is striclty positive and the second order equation associated with \eqref{eq:system_lambda_n_k_second_order_inequality} has one strictly negative and one strictly positive root. 
Furthermore, notice also that matrices $\{L_{n-k}\}_{k\in\mathcal{K}}\in\mathfrak{L}_N${\em -class} and, as said before, in general they are not Laplacian matrices of any graph $\mathcal{G}$. 
The collection of pair-wise simultaneously diagonalizable matrices obtained imposing the iterative constraints \eqref{eq:system_lambda_n_k_positive}\--\eqref{eq:system_lambda_n_k_second_order_inequality} is formalized in the following definition.

\begin{defi}\label{def:N_n_collection}
Given two integers $N,n\in\Bbb{N}$, with $n,N\geq 2$, the collection of matrices  $\{L_{n-k}\}_{k\in\mathcal{K}}\in\mathfrak{L}_N${\em -class}, with $\mathcal{K}=\{0,\dots,n-1\}$, is said to be a $(N,n)${\em -collection} if the matrices are pair-wise simultaneously diagonalizable and satify the iterative spectrum constraints \eqref{eq:system_lambda_n_k_positive}\--\eqref{eq:system_lambda_n_k_second_order_inequality} of Algorithm \ref{alg:spectral_constranit_centralized_info}.
\end{defi}

Notice that, since inequalities \eqref{eq:system_lambda_n_k_positive}\--\eqref{eq:system_lambda_n_k_second_order_inequality}  are always feasible, such collection is never empty.

When a specific interconnection topology $\mathcal{G}$ needs to be taken into account, the more restrictive $(\mathcal{G},n)${\em -collection} can be considered, as it is clear from the following definition. 

\begin{defi}\label{def:G_n_collection}
Given a connected graph $\mathcal{G}$ of $N$ nodes and an integer $n\in\Bbb{N}$, with $n,N\geq 2$, the collection of matrices  $\{L_{n-k}\}_{k\in\mathcal{K}}\in\mathfrak{L}_N${\em -class}, with $\mathcal{K}=\{0,\dots,n-1\}$, is said to be a $(\mathcal{G},n)${\em -collection} if they are a $(N,n)${\em -collection} and $\{L_{n-k}\}_{k\in\mathcal{K}}$ are weighted Laplacian matrices of the graph $\mathcal{G}$.
\end{defi}

For the existence of a $(\mathcal{G},n)${\em -collection} associated to a given connected graph $\mathcal{G}$, the following lemma can be given. 

\begin{lem}
Given a connected graph $\mathcal{G}$ of $N$ nodes and an integer $n\in\Bbb{N}$, with $n,N\geq 2$, there always exists an associated $(\mathcal{G},n)${\em -collection}.
\end{lem}

\begin{proof}
The existence of a $(\mathcal{G},n)${\em -collection} can be proved in a constructive way via Algorithm \ref{alg:spectral_constranit_decentralized_info}.
\end{proof}

\begin{algorithm}
\caption{Spectral constraints assignment for constrained topologies}\label{alg:spectral_constranit_decentralized_info}

\begin{algorithmic}[1]

\State Choose any $L(\mathcal{G})$ which is a compatible weighted Laplacian of any  desired connected graph $\mathcal{G}$ .

\State Set $\scriptstyle  L_n\gets L$ 

\State Set $\scriptstyle \{\lambda_{n}^{(1)},\lambda_{n}^{(2)},\dots, \lambda_{n}^{(N)} \}\gets \mathrm{eig}\{L_{n}\}$

\For{i=2, \dots, N}

\State Set $\scriptstyle  s_{n-1}^{(i)}\gets \lambda_{n}^{{(i)}^2}$ 

\State Set $\scriptstyle  \rho_{n-1}^{(i)}\gets \frac{s_{n-1}^{(i)}}{\lambda_n^{(i)}}$

\EndFor

\State Choose $\scriptstyle  0<\bar{\rho}_{n-1}<\min_{i=2, \dots, N}\rho_{n-1}^{(i)}$

\State Set $\scriptstyle  L_{n-1}\gets \bar{\rho}_{n-1}L_n$

\For{k=2,\dots, n-1}

\State Set $\scriptstyle  \{\lambda_{n-k+1}^{(1)},\lambda_{n-k+1}^{(2)},\dots, \lambda_{n-k+1}^{(N)}\}\gets \mathrm{eig}\{L_{n-k+1}\}$

\For{i=2, \dots, N}

\State Compute $\scriptstyle \beta_{n-k+1}^{(i)}$ 

\State Compute $\scriptstyle  \alpha_{n-k+1}^{(i)}$  

\State Compute $\scriptstyle  \gamma_{n-k}^{(i)}$  

\State Set $\scriptstyle s_{n-k}^{(i)}\gets \min\{r_{n-k,1}^{(i)},r_{n-k,2}^{(i)}\}$, with

\begin{align*}
\scriptstyle  r_{n-k,1}^{(i)}& \scriptstyle = \frac{2\lambda_{n-k+1}^{(i)}\alpha_{n-k+1}^{(i)}}{\gamma_{n-k}^{{(i)}^2}},\\
\scriptstyle r_{n-k,2}^{(i)}&\scriptstyle  = \mathop{\sup}\limits_{r\in\R}\left\{ 
\gamma_{n-k+1}^{{(i)}^2}r^2+8\lambda_{n-k+2}^{(i)}\beta_{n-k+1}^{(i)}r-4\lambda_{n-k+1}^{{(i)}^2}\beta_{n-k+1}^{(i)}<0  \right\}.
\end{align*}

\State Set $\scriptstyle \rho_{n-k}^{(i)}\gets \frac{s_{n-k}^{(i)}}{\lambda_{n-k+1}^{(i)}}$

\EndFor

\State Choose $\scriptstyle 0<\bar{\rho}_{n-k}<\min_{i=2, \dots, N}\rho_{n-k}^{(i)}$

\State Set $\scriptstyle L_{n-k}\gets \bar{\rho}_{n-k}L_{n-k+1}$

\EndFor

\end{algorithmic}
\end{algorithm}

Roughly speaking, the procedure described in Algorithm \ref{alg:spectral_constranit_decentralized_info} allows to obtain $\{L_{n-k}\}_{k\in\mathcal{K}}$ which are weighted Laplacian for any arbitrary connected graph $\mathcal{G}$. Their expression is $L_{n-k}=l_{n-k}L$, where $L=L(\mathcal{G})$ and $l_{n-k}$ is a positive gain defined by the recursive formula $l_{n-k}=\bar{\rho}_{n-k}l_{n-k+1}$, with $l_n=1$. Furthermore, the fact that such matrices are also a $(N,n)${\em -collection} can be trivially showed by noticing that the spectral constraints \eqref{eq:system_lambda_n_k_positive}\--\eqref{eq:system_lambda_n_k_second_order_inequality} are satisfied.

\begin{rem}
It is worth noticing that Algorithm \ref{alg:spectral_constranit_centralized_info} has been introduced specifically to define a  $(N,n)${\em -collection} (and so also the special case of $(\mathcal{G},n)${\em -collection}). The spectral constraints assigned in such an iterative way to the matrices in the collection will be shown to be sufficient for the network synchronization. Notice also that in several papers in the literature, sufficient conditions on the spectrum of the Laplacian matrix of the graph are given in order to prove synchronization, and the same happens in the current paper. However, due to the fact that any possible system degree is here considered, the conditions are given through an iterative procedure rather than using a closed expression. 

It is also worth noticing that the fact that a $(\mathcal{G},n)${\em -collection}) is never empty for any connected graph $\mathcal{G}$ will ensure the solvability of the higher-order free synchronization problem with local controllers.
\end{rem}

\section{Synchronization of systems in companion form}\label{sec:synch_canonical_form}
In this section we give the main results of the paper, i.e., proving that local controllers are able to synchronize a network of nonlinear systems in companion form, as stated in Section \ref{sec:problem_formulation}. Specifically, here we propose both a pure proportional and an integral-proportional controller. 
{It is worth noticing} that, in our approach, the analytic expression of the Lyapunov function that allows to prove the results is parametrized by the system order $n$. Indeed, its expression will be obtained by means of the $(N,n)${\em -collection} generated with Algorithm\ref{alg:spectral_constranit_centralized_info} for any given system order.

\subsection{Synchronization with $PD^{n-1}$ controllers} \label{sec:P_controller}
The following theorem gives conditions on the existence of a solution for the free synchronization problem of dynamical systems in companion form. 

\begin{thm}\label{thm:P_controller_companion_form}
Let us consider $N$ dynamical agents in companion form \eqref{eq:canonical_control_form} and suppose that   $f(t,x^{(i)})$ is weak-Lipschitz with constant $w$. Let us consider a $(N,n)${\em -collection} $\{L_1,\dots,L_n\}$ (or, more specifically, a $(\mathcal{G},n)${\em -collection} associated with a connected graph $\mathcal{G}$). 
Then, the free synchronization problem stated in Section \ref{sec:problem_formulation} is solvable with the following  proportional-derivative controllers
\begin{equation*}
\tilde{u}^{(i)}(t)=l\sum_{k=1}^n\sum_{j=1}^N l_{kij}(x_k^{(j)}(t)-x_k^{(i)}(t)), \quad i=1,\dots,N
\end{equation*}
with $l_{kij}$ being the elements of the matrices $L_k=[l_{kij}]$, with $k=1,\dots,n$, and $l>1$ being a scalar gain satisfying 
\begin{equation}\label{eq:inequality_l}
l>\frac{1}{\tilde{\beta}}(w\bar{\lambda}_{\max}+\tilde{\beta}-\bar{\beta}),
\end{equation}
where in the above expression $\bar{\beta}$, $\bar{\lambda}_{\max}$ and $\tilde{\beta}$ are positive scalars defined respectively as $\bar{\beta}=\min_{i=2,\dots,N}\beta_0^{(i)}$, $\bar{\lambda}_{\max}=\max\mathrm{eig}\{\bar{L}\}$, with $\bar{L}=\sum_{k=1}^n L_k$,  and $\tilde{\beta}=\min_{i=2,\dots,N}\left\{\bar{\beta},\lambda_{n}^{{(i)}^2}\right\}$. 
\end{thm} 

\begin{proof}
The proof of the above result is obtained by constructing a suitable Lyapunov function for the synchronization error trajectory able to exploit the specific canonical structure. To do so, we will divide the proof in two steps. In the first one we will define appropriate matrices upon which we will derive a candidate Lyapunov function. In the second part we will define the stack error system and we will prove the stability by means of such an obtained function. 
\newline
\textbf{Part 1: Definition of appropriate matrices.} Let us denote for convenience $L_{n+1}=1/2\cdot I_N$, $L_0=O_N$, and let us consider the positions $\lambda_{n+1}^{(i)}=1/2$ and $\lambda_0^{(i)}=0$.
We define the matrices $\{M_{n-k}\}_{k\in\mathcal{K}}$, with $M_{n-k}\in\R^{(k+1)N\times(k+1)N}$,  in the following recursive way
\begin{equation}\label{eq:matrix_M_n_k}
M_{n-k}=\left[
\begin{array}{cc}
M_{\varphi,n-k} & M_{\psi,n-k}\\
M^T_{\psi,n-k} & M_{n-k+1}
\end{array}
\right],
\end{equation}
with $M_{\varphi,n-k}=2L_{n-k}L_{n-k+1}$ and $M_{\psi,n-k}=\left[2L_{n-k}L_{n-k+2}, \dots, 2L_{n-k}L_{n}, 2L_{n-k}L_{n+1}\right]$, and where as terminal condition of the recursion we define $M_n=L_n$. 
It is easy to notice from the above definition that matrices $\{M_{n-k}\}_{k\in\mathcal{K}}$ are  $(k+1)\times (k+1)$ symmetric block matrices.

Analogously, we consider the $\{M_{n-k}^{(i)}\}_{(i,k)\in\mathcal{I}\times\mathcal{K}}$ matrices, with $M_{n-k}^{(i)}\in\R^{(k+1)\times (k+1)}$ and with $\mathcal{I}=\{2,\dots,N\}$, recursively defined as
\begin{equation}\label{eq:matrix_M_n_k_i}
M_{n-k}^{(i)}=\left[
\begin{array}{cc}
M_{\varphi,n-k}^{(i)} & M_{\psi,n-k}^{(i)}\\
{M^T}_{\psi,n-k}^{(i)} & M_{n-k+1}^{(i)}
\end{array}
\right],
\end{equation}
with $M_{\varphi,n-k}^{(i)}=2\lambda_{n-k}^{(i)}\lambda_{n-k+1}^{(i)}$, $M_{\psi,n-k}^{(i)}=\left[2\lambda_{n-k}^{(i)} \lambda_{n-k+2}^{(i)}, \dots, 2\lambda_{n-k}^{(i)}\lambda_{n}^{(i)}, 2\lambda_{n-k}^{(i)}\lambda_{n+1}^{(i)}\right]$, and with $M_n^{(i)}=\lambda_n^{(i)}$.

Together with matrices $\{M_{n-k}\}_{k\in \mathcal{K}}$ and $\{M_{n-k}^{(i)}\}_{(i,k)\in\mathcal{I}\times\mathcal{K}}$, we also define the symmetric matrices $\{H_{n-k}\}_{k\in \mathcal{K}}$, with $H_{n-k}\in\R^{(k+1)N\times(k+1)N}$ and $\{H_{n-k}^{(i)}\}_{(i,k)\in\mathcal{I}\times\mathcal{K}}$, with $H_{n-k}^{(i)}\in\R^{(k+1)\times(k+1)}$. Specifically,
\begin{equation}\label{eq:matrix_H_n_k}
H_{n-k}=\left[
\begin{array}{cc}
H_{\varphi,n-k} & H_{\psi,n-k}\\
H^T_{\psi,n-k} & H_{n-k+1}
\end{array}
\right],
\end{equation}
with $H_{\varphi,n-k}=L_{n-k}^2-2L_{n-k-1}L_{n-k+1}$, $H_{\psi,n-k}=\left[-L_{n-k-1}L_{n-k+2}, \dots, -L_{n-k-1}L_{n},-L_{n-k-1}L_{n+1}\right]$, and with  $H_n=L_n^2-L_{n-1}$, while $H_{n-k}^{(i)}$ is defined as
\begin{equation}\label{eq:matrix_H_n_k_i}
H_{n-k}^{(i)}=\left[
\begin{array}{cc}
H_{\varphi,n-k}^{(i)} & H_{\psi,n-k}^{(i)}\\
{H^T}_{\psi,n-k}^{(i)} & H_{n-k+1}^{(i)}
\end{array}
\right],
\end{equation}
with $H_{\varphi,n-k}^{(i)}=\lambda_{n-k}^{{(i)}^2}-2\lambda_{n-k-1}^{(i)}$, $H_{\psi,n-k}^{(i)}=\left[-\lambda_{n-k-1}^{(i)}\lambda_{n-k+2}^{(i)}, \dots, -\lambda_{n-k-1}^{(i)}\lambda_{n}^{(i)}, -\lambda_{n-k-1}^{(i)}\lambda_{n+1}^{(i)}\right]$, and with $H_n^{(i)}=\lambda_n^{{(i)}^2}-\lambda_{n-1}^{(i)}$.

From the above definitions it is immediate to see that $y^TM_1y=0$ and $y^TH_1y=0$, for all $y\in \Delta$. We are now going to prove that, for all $y\in \Delta^{\perp}-\{0\}$, i.e. for all the vector orthogonal to the synchronization manifold, we have  $y^TM_1y>0$ and $y^TH_1y>0$. This fact will be a key aspect later, where we will derive a Lyapunov function for the system.

First, let us consider the set of vectors 
\begin{align*}
S_{\Delta^{\perp}}=&\left\{
\varepsilon_1\otimes v^{(2)},\dots, \varepsilon_1\otimes v^{(N)},
\varepsilon_2\otimes v^{(2)},\dots, \varepsilon_2\otimes v^{(N)},\right.\\
& \left. \quad \dots,
\varepsilon_n\otimes v^{(2)},\dots, \varepsilon_n\otimes v^{(N)}
\right\},
\end{align*}
with $\varepsilon_i\in\R^n$ being the vector with a unitary entry in the $i$-th position and all the other entries null.

It is easy to see that $S_{\Delta^{\perp}}\subset\R^{nN}$ is a set of orthogonal unitary vectors and that $\Delta^{\perp}=span\{S_{\Delta^{\perp}}\}$. Hence, any vector $y\in\Delta^{\perp}$ can be expressed as a liner combination of the vectors in $S_{\Delta^{\perp}}$ or, more compactly, it can be expressed as $y=\sum_{i=2}^N y^{(i)}$, where $y^{(i)}=c^{(i)}\otimes v^{(i)}$ and where $c^{(i)}=(c_1^{(i)},\dots, c_n^{(i)})^T\in\R^n$ is a vector of coefficients. 

Now, due to the orthogonality of $v^{(i)}$ and $v^{(j)}$, we have that, for all $i\neq j$, ${y^{(j)}}^TM_1y^{(i)}=0$ and ${y^{(j)}}^TH_1y^{(i)}=0$, while remembering definitions \eqref{eq:matrix_M_n_k_i} and \eqref{eq:matrix_H_n_k_i}  we have ${y^{(i)}}^TM_1y^{(i)}={c^{(i)}}^TM_1^{(i)}c^{(i)}$ and ${y^{(i)}}^TH_1y^{(i)}={c^{(i)}}^TH_1^{(i)}c^{(i)}$. So,
\begin{equation}\label{eq:quadtratic_form_M1}
y^TM_1y=\sum_{i=2}^N{c^{(i)}}^TM_1^{(i)}c^{(i)}, 
\end{equation}
and 
\begin{equation}\label{eq:quadtratic_form_H1}
y^TH_1y=\sum_{i=2}^N{c^{(i)}}^TH_1^{(i)}c^{(i)}.
\end{equation}
Now, guaranteeing that ${c^{(i)}}^TM_1^{(i)}c^{(i)}>0$ and ${c^{(i)}}^TH_1^{(i)}c^{(i)}>0$, for all $c^{(i)}\in\R^{n}-\{0\}$ and for all $i\in \mathcal{I}$, implies the strict positivity of \eqref{eq:quadtratic_form_M1} and \eqref{eq:quadtratic_form_H1}, respectively. For this reason, the rest of this first part of the proof {is devoted to showing} the positive definiteness of matrices $M_1^{(i)}$ and $H_1^{(i)}$. Specifically, we first focus on proving the positivity of $M_1^{(i)}$ via an induction argument which exploits the recursive structure of the matrix itself. 
First of all, we can see that the trailing principal submatrix 
\[
M_{n-1}^{(i)}=
\left[
\begin{array}{cc}
2\lambda_{n-1}^{(i)}\lambda_{n}^{(i)} & \lambda_{n-1}^{(i)}\\
\lambda_{n-1}^{(i)} & \lambda_{n}^{(i)}
\end{array}
\right],
\]
is {positively defined}. Indeed, the Sylvester's criterion can be applied since $\lambda_n^{(i)}>0$ and its determinant is positive due to the choice $\lambda_{n-1}^{(i)}<\lambda_{n}^{{(i)}^2}$ (initialization of Algorithm \ref{alg:spectral_constranit_centralized_info}). So, trivially we have that $\alpha_{n-1}^{(i)}>0$ and, since $M_{n-1}^{(i)}=A_{n-1}^{(i)}$, the relation 
$
z^TM_{n-1}^{(i)}z\geq z^TA_{n-1}^{(i)}z\geq \alpha_{n-1}^{(i)}z^Tz
$
holds for all $z\in \R^2$.
Furthermore, $\gamma_{n-1}>0$ trivially holds.
For the induction argument, we suppose that the same relation holds for a generic $M_{n-k+1}^{(i)}$, with $k\geq 2$, namely
\begin{equation}\label{eq:inequality_M_n_k_plus_1}
z^TM_{n-k+1}^{(i)}z\geq z^TA_{n-k+1}^{(i)}z\geq \alpha_{n-k+1}^{(i)}z^Tz, \quad\forall z\in \R^k,
\end{equation}
with $\alpha_{n-k+1}^{(i)}>0$. We also suppose that $\gamma_{n-k}^{(i)}>0$. 
With such an assumption, we study the quadratic form $\bar{z}_{k}^TM_{n-k}\bar{z}_{k}$, for all the vectors $\bar{z}_{k}\in \R^{k+1}-\{0\}$, and where we have defined $\bar{z}_{k}=(z_1,\dots,z_{k+1})^T$. For convenience, we introduce the subvector $\bar{z}_{k-1}$ of the last $k$ elements of $\bar{z}_k$, and so, in block form, we have $\bar{z}_{k}=[z_1|\bar{z}_{k-1}^T]^T$. We obtain
\begin{align*}
\bar{z}_{k}^TM_{n-k}\bar{z}_{k}=&2\lambda_{n-k}^{(i)}\lambda_{n-k+1}^{(i)}z_1^2+\sum_{j=2}^{k}4\lambda_{n-k}^{(i)}\lambda_{n-k+j}^{(i)}z_1z_j+\\
& 2\lambda_{n-k}^{(i)}z_1z_{k+1}+\bar{z}_{k-1}^TM_{n-k+1}^{(i)}\bar{z}_{k-1}.
\end{align*}
Considering now $z_1z_h=\min_{j=2,\dots,k+1}z_1z_j, $
and remembering inequality \eqref{eq:inequality_M_n_k_plus_1}, we obtain 
\begin{align*}
\bar{z}_{k}^TM_{n-k}\bar{z}_{k} \geq & 2\lambda_{n-k}^{(i)}\lambda_{n-k+1}^{(i)}z_1^2+\\
 & 2\left[1+\sum_{j=2}^{k}2\lambda_{n-k+j}^{(i)}
\right]\lambda_{n-k}^{(i)}z_1z_h+\alpha_{n-k+1}^{(i)}z_h^2\\
=& 2\lambda_{n-k}^{(i)}\lambda_{n-k+1}^{(i)}z_1^2+2\gamma_{n-k}^{(i)}\lambda_{n-k}^{(i)}z_1z_h+\\
 & \alpha_{n-k+1}^{(i)}z_h^2.
\end{align*}
Now, considering the definition of $A_{n-k}^{(i)}$, it is immediate to notice that the quadratic expression above can be written as $[z_i,z_h]A_{n-k}^{(i)}[z_i,z_h]^T$. So, its positivity is guaranteed if and only if the matrix $A_{n-k}^{(i)}$ is {positively defined}. Since $\alpha_{n-k+1}^{(i)}>0$, and since condition \eqref{eq:system_lambda_n_k_first_order_inequality} in 
Algorithm \ref{alg:spectral_constranit_centralized_info} imposes the positivity of the determinant of $A_{n-k}^{(i)}$, applying again the Sylvester's criterion we conclude that $A_{n-k}^{(i)}>0$. 
Iterating the reasoning for all $k=2,\dots,n-1$ we obtain $M_1^{(i)}>0$.
 
An analogous reasoning can be adopted to prove positive definiteness of $H_1^{(i)}$. Indeed, it is immediate to see that the trailing principal submatrix $H_{n}^{(i)}\in\R^{1\times 1}$ is positive since $H_{n}^{(i)}=\beta_{n-1}^{(i)}=\lambda_n^{{(i)}^2}-\lambda_{n-1}^{(i)}>0$, again for the initial choice $0<\lambda_{n-1}^{(i)}<\lambda_n^{{(i)}^2}$. Obviously, the relation 
\[
z^TH_n^{(i)}z\geq\beta_{n-1}^{(i)}z^Tz
\]
holds for all $z\in\R$.
As done for $M_{n-k}^{(i)}$, also for proving the positive definiteness of $H_{n-k}^{(i)}$ an induction argument will be used. To do so, we suppose
\begin{equation}\label{eq:inequality_N_n_k_plus_1}
z^T H_{n-k+1}^{(i)}z\geq \beta_{n-k}z^T z, \quad \forall z\in \R^{k},
\end{equation}
with $\beta_{n-k}>0$. Furthermore, from the iterative reasoning applied for proving that $M_1^{(i)}>0$, we implicitly obtained that $\gamma_{n-k}^{(i)}>0$ for all $k=1, \dots, n-1$, since $\lambda_{n-k}^{(i)}>0$ for all $k=1, \dots, n-1$.
Defining $\bar{z}_{k}$ as before, we can write the quadratic form $\bar{z}_{k}^{T}H_{n-k}^{(i)}\bar{z}_{k}$, for all $\bar{z}_{k}\in\R^{k+1}-\{0\}$,   as
\begin{align*}
\bar{z}_{k}^{T}H_{n-k}^{(i)}\bar{z}_{k}=&\left[\lambda_{n-k}^{{(i)}^2}-2\lambda_{n-k-1}^{(i)}\lambda_{n-k+1}^{(i)}\right]z_1^2-\\
& \sum_{j=2}^{k}2\lambda_{n-k-1}^{(i)}\lambda_{n-k+j}z_1z_j-\\
& \lambda_{n-k-1}^{(i)}z_1z_{k+1}+\bar{z}_{k-1}^{T}H_{n-k+1}^{(i)}\bar{z}_{k-1}. 
\end{align*}
Considering $z_1z_h=\max_{j=2,\dots,k+1}z_1z_j,$ and taking into account \eqref{eq:inequality_N_n_k_plus_1}, we obtain the following inequality
\begin{align*}
\bar{z}_{k}^{T}H_{n-k}^{(i)}\bar{z}_{k}\geq &
\left[\lambda_{n-k}^{{(i)}^2}-2\lambda_{n-k-1}^{(i)}\lambda_{n-k+1}^{(i)}\right]z_1^2- \\
& \left[1+\sum_{j=2}^k 2\lambda_{n-k+j}^{(i)}\right]\lambda_{n-k-1}^{(i)}z_1z_h+\beta_{n-k}^{(i)}z_h^2\\
=&\left[\lambda_{n-k}^{{(i)}^2}-2\lambda_{n-k-1}^{(i)}\lambda_{n-k+1}^{(i)}\right]z_1^2-\\
&\gamma_{n-k}^{(i)}z_1z_h+\beta_{n-k}^{(i)}z_h^2.
\end{align*}
Observing that the above quadratic form can be obtained from $[z_i,z_h]B_{n-k-1}^{(i)}[z_i,z_h]^T$, since $\beta_{n-k}^{(i)}>0$ for the Sylvester's criterion the positive definiteness of $B_{n-k-1}$ is guaranteed by the positivity of its determinant. The latter condition is given by \eqref{eq:system_lambda_n_k_second_order_inequality} of Algorithm \ref{alg:spectral_constranit_centralized_info} evaluated at $k+1$. Repeating the reasoning for $k=1,\dots,n$ we obtain $z^T H_{1}^{(i)}z \geq \beta_0^{(i)}z^T z$, with $\beta_0^{(i)}>0$, which guarantees positive definiteness of $H_{1}^{(i)}$. It is also possible to further analyze the quadratic form \eqref{eq:quadtratic_form_H1}, as this will turn useful later in Step 2 of the proof. For all $y\in\Delta^\perp-\{0\}$ we have, 
\begin{align}
y^TH_1y=\sum_{i=2}^N{c^{(i)}}^TH_1^{(i)}c^{(i)}&\geq\sum_{i=2}^N\beta_{0}^{(i)}{c^{(i)}}^Tc^{(i)}\nonumber \\
& \geq \bar{\beta}\sum_{i=2}^N{c^{(i)}}^Tc^{(i)}\geq  \bar{\beta}y^T y,\label{eq:inequality_N_1}
\end{align}
where $\bar{\beta}=\min_{i=2,\dots,N}\beta_0^{(i)}$ is a positive scalar and where we considered 
$
{y^{(i)}}^Ty^{(i)}=\left[{c^{(i)}}^T\otimes {v^{(i)}}^T\right]\left[{c^{(i)}}\otimes {v^{(i)}}\right]
= {c^{(i)}}^T{c^{(i)}},
$
and where ${y^{(i)}}^T{y^{(j)}}=0$, for $i\neq j$.
\newline
\textbf{Part 2: Lyapunov stability analysis.} For convenience we consider the error stack system of the form  
\begin{eqnarray}\label{eq:stack_error_form}
\dot{e}_1&=&e_2 \nonumber \\
&\vdots & \nonumber \\
\dot{e}_n&=&F(t,x)-\bar{f}(t,x)\cdot 1_N+\tilde{u}(t), \nonumber\\
\end{eqnarray}
where $\bar{f}(t,x)=1/N\sum_{j=1}^N f(t,x^{(j)})$ and with 
$
\tilde{u}(t)=-l\sum_{k=1}^n L_ke_k(t),
$
where $L_k$, with $k=1,\dots,n$, are given in the theorem statement.
Remembering the definition of matrix $M_1$ in \eqref{eq:matrix_M_n_k} with $k=n-1$, we can also rewrite it in the block form
\begin{equation*}
M_{1}=\left[
\begin{array}{cc}
M_{\vartheta} & M_{\varsigma}\\
M^T_{\varsigma} & L_n
\end{array}
\right],
\end{equation*}
with $M_{\vartheta}\in\R^{(n-1)N\times (n-1)N}$ leading principal submatrix. For the error system \eqref{eq:stack_error_form} we can finally consider the quadratic candidate Lyapunov function\footnote{The explicit dependence on $n$ of the Lyapunov function points out that the matrix $M_1$, {from which} $\mathop M\limits^{\sim}\in\R^{nN\times nN}$ is derived, has a specific structure depending on the system order $n$ considered.} $V(e,n)=1/2 e^T \mathop M\limits^{\sim}e$, where $\mathop M\limits^{\sim}\in\R^{nN\times nN}$ is defined from $M_1$ by considering as leading principal submatrix $lM_{\vartheta}$, while all the other submatrices are the same as in $M_1$, i.e.,
\begin{equation}\label{eq:M_tilde}
\mathop M\limits^{\sim}=\left[
\begin{array}{cc}
lM_{\vartheta} & M_{\varsigma}\\
M^T_{\varsigma} & L_n
\end{array}
\right],
\end{equation}
It easy to see that such quadratic form is a valid candidate Lyapunov function for proving synchronization since $y^T \mathop M\limits^{\sim}y=0$ for all $y\in\Delta$, while $y^T \mathop M\limits^{\sim}y>0$ for all $y\in\Delta^\perp-\{0\}$. The first property follows immediately from the definition, while the latter can be shown partitioning the generic $y$ as $y=[y_\vartheta^T,y_\varsigma^T]^T$ and considering 
$y^T \mathop M\limits^{\sim}y=y^TM_1y+(l-1)y_{\vartheta}^TM_\vartheta y_{\vartheta}$. The positivity is so proved remembering that $M_1$ is positive definite on $\Delta^\perp-\{0\}$, as showed in Part 1, while $M_\vartheta$ is its leading principal minor and is, therefore, positive. 
Considering the time derivative of $V(e,n)$ we obtain 
\begin{equation}\label{eq:V_dot_step1}
\dot{V}(e,n)=e^T \mathop M\limits^{\sim} \dot{e}= e^T\mathop M\limits^{\sim}\Phi(t,x)+e^T\mathop M\limits^{\sim}\Xi(e),
\end{equation}
with $\Phi(t,x)=\left[0_N^T, \dots, 0_N^T, F^T(t,x)-\bar{f}(t,x)\cdot 1_N^T\right]^T$ and $\Xi(e)=\left[e_2^T, \dots,e_n^T, -\left(\sum_{k=1}^n L_k e_k(t)\right)^T \right]^T$.
We now analyze separately the two terms in \eqref{eq:V_dot_step1}. For the first one we have 
\begin{align*}
e^T\mathop M\limits^{\sim}\Phi(t,x)=&\sum_{k=1}^n e_k^TL_k\left[F(t,x)-\bar{f}(t,x)\cdot 1_N\right]\\
=&\sum_{k=1}^n \frac{1}{2}\sum_{i=1}^N\sum_{j=1}^N \\	
& l_{kij}\left[e_k^{(i)}-e_k^{(j)}\right]\left[f(t,x^{(i)})-f(t,x^{(j)})\right]\\
=&\sum_{k=1}^n \frac{1}{2}\sum_{i=1}^N\sum_{j=1}^N \\
& l_{kij}\left[x_k^{(i)}-x_k^{(j)}\right]\left[f(t,x^{(i)})-f(t,x^{(j)})\right],
\end{align*}
from which, using the weak-Lipschitz property
\begin{align}
e^T\mathop M\limits^{\sim}\Phi(t,x)\leq &
\sum_{k=1}^n \frac{1}{2}\sum_{i=1}^N\sum_{j=1}^N \\
& l_{kij} w \left[x^{(i)}-x^{(j)}\right]^T\left[x^{(i)}-x^{(j)}\right] \nonumber \\
=& \sum_{k=1}^n w \sum_{h=1}^n \frac{1}{2}\sum_{i=1}^N\sum_{j=1}^N l_{kij}\left[e_h^{(i)}-e_h^{(j)}\right]^2 \nonumber \\
=& \sum_{k=1}^n w \sum_{h=1}^n e_h^T L_k e_h= \sum_{k=1}^n w e^T (I_n \otimes L_k) e \nonumber \\
=& w e^T \left(I_n \otimes  \sum_{k=1}^n  L_k \right) e = w e^T \left(I_n \otimes \bar{L}\right) e.\label{eq:inequality_weak_Lipschitz_global_dynamics}
\end{align}

For the analysis of the second term in \eqref{eq:V_dot_step1}, we first write matrix $H_1$ in a  block form analogous to $M_1$, namely
\begin{equation*}
H_{1}=\left[
\begin{array}{cc}
H_{\vartheta} & H_{\varsigma}\\
H^T_{\varsigma} & L_n^2-L_{n-1}
\end{array}
\right].
\end{equation*}

From the above matrix we define $\mathop H\limits^{\sim}\in\R^{nN\times nN}$ as 
\begin{equation*}
\mathop H\limits^{\sim}=\left[
\begin{array}{cc}
lH_{\vartheta} & H_{\varsigma}\\
H^T_{\varsigma} & lL_n^2-L_{n-1}
\end{array}
\right].
\end{equation*}

Now, performing suitable algebraic manipulations, we can show that 
\begin{equation}\label{eq:equality_M_tilde_H_tilde}
e^T\mathop M\limits^{\sim}\Xi(e)=-e^T \mathop H\limits^{\sim} e.
\end{equation}

To do so, we take advantage of the recursive structure of the matrices $M_1$ and $H_1$, respectively obtained nesting  \eqref{eq:matrix_M_n_k} and \eqref{eq:matrix_H_n_k} up to index $k=n-1$. Remembering that $L_{n+1}=1/2 \cdot I_N$ and $L_0=O_N$, we have that $\mathop M\limits^{\sim}={\mathop M\limits^{\sim}}_1$, with ${\mathop M\limits^{\sim}}_1$ defined nesting up to $k=n-1$ the following 

\begin{equation*}   
{\mathop M\limits^{\sim}}_{n-k}=\left[
\begin{array}{cc}
{\mathop M\limits^{\sim}}_{\varphi,n-k} & {\mathop M\limits^{\sim}}_{\psi,n-k}\\
{\mathop M\limits^{\sim}}^T_{\psi,n-k} & {\mathop M\limits^{\sim}}_{n-k+1}
\end{array}
\right],
\end{equation*}
with ${\mathop M\limits^{\sim}}_{\varphi,n-k}=2lL_{n-k}L_{n-k+1}$ and ${\mathop M\limits^{\sim}}_{\psi,n-k}=\left[2lL_{n-k}L_{n-k+2}, \dots, 2lL_{n-k}L_{n}, L_{n-k}\right]$, and where as terminal condition of the recursion we define ${\mathop M\limits^{\sim}}_n=L_n$. 

Analogously, we have $\mathop H\limits^{\sim}={\mathop H\limits^{\sim}}_1$, with ${\mathop H\limits^{\sim}}_1$ defined nesting up to $k=n-1$ the following 
\begin{equation*}   
{\mathop H\limits^{\sim}}_{n-k}=\left[
\begin{array}{cc}
{\mathop H\limits^{\sim}}_{\varphi,n-k} & {\mathop H\limits^{\sim}}_{\psi,n-k}\\
{\mathop H\limits^{\sim}}^T_{\psi,n-k} & {\mathop H\limits^{\sim}}_{n-k+1}
\end{array}
\right],
\end{equation*}
with ${\mathop H\limits^{\sim}}_{\varphi,n-k}=lL_{n-k}^2-2lL_{n-k-1}L_{n-k+1}$, ${\mathop H\limits^{\sim}}_{\psi,n-k}=\left[-lL_{n-k-1}L_{n-k+2}, \dots, -lL_{n-k-1}L_{n},-\frac{1}{2}L_{n-k-1}\right]$, and with  ${\mathop H\limits^{\sim}}_n=lL_n^2-L_{n-1}$.

Relation \eqref{eq:equality_M_tilde_H_tilde} can be proved focusing on a generic trail principal submatrix ${\mathop H\limits^{\sim}}_{n-k}$ of ${\mathop H\limits^{\sim}}_1$. In particular, we restrict our attention on the first row and column of submatrix ${\mathop H\limits^{\sim}}_{n-k}$. The associated terms  will be involved in the bilinear terms $e_i^T\eta_{ij}e_j$ with $i=n-k$ and $j=n-k, \dots, n$ and with $i=n-k, \dots, n$ and $j=n-k$, where with $\eta_{ij}$ we have here denoted the ${i,j}$-th entry of matrix ${\mathop H\limits^{\sim}}_1$, i.e. ${\mathop H\limits^{\sim}}_1=[\eta_{ij}]$.

From the definition of matrix ${\mathop H\limits^{\sim}}_1$, it is easy to see that the terms in  $e^T{\mathop H\limits^{\sim}}_1\Xi(e)$ corresponding to the bilinear terms $e_i^T\eta_{ij}e_j$ considered, are given by 
\begin{align*}
-\sum_{{i=n-k,j=n-k, \dots, n }\atop{i=n-k, \dots, n,j=n-k}} {e_i^T\eta_{ij}e_j}  &=
2le_{n-k}^{T}L_{n-k-1}L_{n-k+1}e_{n-k}^{T}+ \\
& \sum_{j=n-k+1}^{n}2le_{n-k}^{T}L_{n-k}L_je_j-\\
& \sum_{j=n-k}^{n}le_{n-k}^{T}L_{n-k}L_je_j-\\
& \sum_{j=n-k+1}^{n}le_{i}^{T}L_iL_{n-k}e_{n-k}+ \\
& \sum_{i=n-k+1}^{n-1}2le_i^TL_{n-k-1}L_{i+1}e_{n-k}+ \\
& e_n^TL_{n-k-1}e_{n-k},
\end{align*}
from which we obtain
\begin{align*}
-\sum_{{i=n-k,j=n-k, \dots, n }\atop{i=n-k, \dots, n,j=n-k}} {e_i^T\eta_{ij}e_j} &=
2e_{n-k}^{T}L_{n-k-1}L_{n-k+1}e_{n-k}^{T}-\\
&e_{n-k}^TL_{n-k}^2e_{n-k}+\\
& \frac{1}{2}\sum_{i=n-k+1}^{n-1}\left[e_i^TL_{n-k-1}L_{i+1}e_{n-k}+\right. \\
& \left. e_{n-k}^TL_{n-k-1}L_{i+1}e_i\right]+\\
&\frac{1}{2}e_n^TL_{n-k-1}e_{n-k}+\\
&\frac{1}{2}e_{n-k}^TL_{n-k-1}e_n.
\end{align*}
 
Repeating the same reasoning for all $k\in\{0,\dots, n-1\}$ we finally have \eqref{eq:equality_M_tilde_H_tilde}.
Writing $e=[e_\vartheta^T,e_\varsigma^T]^T$ we have 
$e^T \mathop H\limits^{\sim} e=e^TH_1e+(l-1)e_\vartheta^T H_\vartheta e_\vartheta+(l-1)e_\varsigma^T L_n^2e_\varsigma$ and so, remembering \eqref{eq:inequality_N_1}, the following inequality holds
\begin{equation}\label{eq:inequality_N_tilde}
e^T \mathop H\limits^{\sim}e \geq \bar{\beta}e^Te+(l-1)\tilde{\beta}e^Te, \quad \forall e\in\Delta^{\perp}-\{0\}.
\end{equation}

Combining \eqref{eq:inequality_weak_Lipschitz_global_dynamics} and \eqref{eq:inequality_N_tilde}, from \eqref{eq:V_dot_step1} the following inequality holds
\begin{align*}
\dot{V}(e,n)&\leq w e^T \left(I_n \otimes \bar{L}\right) e-\bar{\beta}e^Te-(l-1)\tilde{\beta}e^Te\\
&\leq w\bar{\lambda}_{\max}e^Te-\bar{\beta}e^Te-(l-1)\tilde{\beta}e^Te. 
\end{align*}

Imposing $w\bar{\lambda}_{\max}-\bar{\beta}-(l-1)\tilde{\beta}<0$ condition \eqref{eq:inequality_l} is obtained which guarantees, together with $l>1$, a negative quadratic upper bound for $\dot{V}(e,n)$ and so the synchronization of the agents to the same trajectory. 
\end{proof}

\begin{rem}
It is worth noticing that the relevant case of consensus of double \cite{re:08,lixi:13} and higher-order \cite{yuch:11} integrators is included in the previous study as a particular case when $f(t,x^{(i)})=0$, and can be studied following exactly the same way of constructing the quadratic Lyapunov function $V(e,2)=1/2 e^T \mathop M\limits^{\sim}e$, with $\mathop M\limits^{\sim}$ given in \eqref{eq:M_tilde}. Specifically, for the consensus of double integrators, the matrix $\mathop M\limits^{\sim}$ can be easily showed to be given by
\begin{equation*}   
\left[
\begin{array}{cc}
2lL_1L_2 & L_1\\
L_1 & L_2
\end{array}
\right].
\end{equation*}
Notice also that, in our study, we directly consider in Algorithm \ref{alg:spectral_constranit_centralized_info} and in Algorithm \ref{alg:spectral_constranit_decentralized_info} at least a second-order degree, i.e., $n \geq 2$, for the interconnected agents. 
In principle, a first order case could still be studied observing that the $n\times n$ block matrix $M_1$, and so also matrix $\mathop M\limits^{\sim}$,  grows in size accordingly to the degree $n$ of the agents from the bottom-right corner $L_n$, thus resulting in the specific recursive structure we highlighted. The case of $n=1$ would so result in the bottom-right corner only, thus having $\mathop M\limits^{\sim}=lL$, which gives a well known Lyapunov function for studying the classical problem of consensus for single integrators \cite{olfa:07}. 
{Also, when considering the higher-order consensus problem and a $(\mathcal{G},n)$-collection is chosen, the controller shows an analogous structure to the one in \cite{yuch:11}. However, in that paper, a different criterion based on the Kharitonov' s theorem is provided in order to select the feedback coefficients $l_{n-k}$.}

Furthermore, together with the case of consensus of integrators,  the relevant case of synchronization of linear systems can be also addressed  with our framework, as will be shown later in Corollary \ref{thm:synch_controllable_linear_systems}.
\end{rem}

\subsection{Synchronization with $PI^hD^{n-1}$ controllers}\label{sec:PI_controller}
The analysis conducted in Section \ref{sec:P_controller}, where a state proportional control action is used to achieve free synchronization, is now extended to the case where an integral control action of any arbitrary degree $h\geq 1$, with $h\in\Bbb{N}$ is also considered. 
 
More in detail, considering a generic integrable function $\eta(\cdot): \R\mapsto\R^n$, we define its integral of degree $h\in\Bbb{N}$ with the following notation
\begin{equation*}
\int_{0}^{t,(h)}\eta(\tau) d\tau := \int_{0}^t \int_{0}^{\tau,(h-1)}\eta(\tau^\prime) d\tau^\prime d\tau \qquad \mathrm{if}\,\, h>1,
\end{equation*}
while in case $h=1$, we simply have
\begin{equation*}
\int_{0}^{t,(1)}\eta(\tau) d\tau := \int_{0}^t \eta(\tau)  d\tau.
\end{equation*}
We now give the following theorem.

\begin{thm}\label{thm:PI_controller_companion_form}
Let us consider $N$ dynamical agents in companion form and suppose that   $f(t,x^{(i)})$ is weak-Lipschitz with constant $w$.  Then, the free synchronization problem is solvable with a $PI^hD^{n-1}$ controller of arbitrary degree $h\geq 1$ of the form
\begin{align*}
\tilde{u}^{(i)}(t)=&l\sum_{k=1}^n\sum_{j=1}^N l_{PD, kij}(x_k^{(j)}(t)-x_k^{(i)}(t)) +\\
& l\sum_{m=1}^h\sum_{j=1}^N l_{I, mij}\int_0^{t,(m)}(x_1^{(j)}(\tau)-x_1^{(i)}(\tau)) d \tau,
\end{align*}
with  $i=1,\dots,N$.
Furthermore, the gain $l$ and the matrices $L_{PD,k}=[l_{PD, kij}]$, with $k=1,\dots, n$, and $L_{I,m}=[l_{I, m ij}]$, with $m=1,\dots, h$, can be selected analogously to Theorem \ref{thm:P_controller_companion_form} considering the following position 
\begin{align}
&L_{I, h-\vartheta +1}  = L_\vartheta, & \vartheta=1,\dots, h \label{eq:L_I_ridefine}\\
&L_{PD, \vartheta -h }  = L_\vartheta ,  & \vartheta=h+1,\dots, h+n\label{eq:L_P_ridefine},
\end{align}
with $\{L_1,\dots,L_{h+n}\}$ being a $(N,h+n)${\em -collection} (or a $(\mathcal{G},h+n)${\em -collection} with $\mathcal{G}$ any connected graph).
\end{thm}
\begin{proof}
The proof of Theorem \ref{thm:PI_controller_companion_form} is given in the Appendix.
\end{proof}

\begin{rem}
The previous result extends Theorem \ref{thm:P_controller_companion_form} allowing an additional integral control action of any degree. {The benefits of integral control actions in disturbance rejection are well known in the literature. Therefore, such additional degree of freedom can be usefully exploited for this aim, as shown in the numerical examples section.}
\end{rem}

\section{Synchronization under canonical transformation}\label{sec:synch_canonical_transformation}
The results stated in Section \ref{sec:synch_canonical_form} can be extended to the relevant class of dynamical systems admitting a canonical control transformation. Roughly speaking, for general nonlinear systems, it suffices to find a nonlinear state transformation $z(t)=T(x(t))$ and apply the $PI^hD^{n-1}$ control law of Theorem \ref{thm:PI_controller_companion_form} to such transformed state. The computation of this nonlinear transformation under suitable involutivity condition of the nonlinear vector field is a well known result in nonlinear control and can be found in \cite{slli:91}.  Also, when the special case of linear systems is considered, the canonical control transformation can be found in \cite{lu:79} and represents a fundamental result in control theory. 

In this section, we first analyse the general case of nonlinear systems, and later the case of linear systems as a separate result. Notice that for the sake of simplicity in the notation, we will consider only time-independent systems. However, analogous results hold for the case of time-dependent systems.

\begin{thm}\label{thm:synch_controllable_nonlinear_systems}
Let us consider a connected graph $\mathcal{G}$ and a multi-agent system of nonlinear dynamical agents of the form
\begin{equation}\label{eq:nonlinear_agent}
\dot{x}^{(i)}=f(x^{(i)})+g(x^{(i)})u^{(i)},\qquad i=1,\dots,N,
\end{equation}
with $x^{(i)}\in\R^n$ and $u^{(i)}\in\R$. Suppose that, for all $x^{(i)}\in\R^{n}$, the following conditions hold:
\begin{enumerate}
\item [(i)] The vectors $\left\{g, \mathrm{ad}_fg,\dots, \mathrm{ad}_f^{n-1}g \right\}$ are linearly independent;
\item [(ii)] The set $\left\{g, \mathrm{ad}_fg,\dots, \mathrm{ad}_f^{n-2}g \right\}$ is involutive;
\item[(iii)] The function $\mathcal{L}_f^{n}\left(T^{-1}(\xi)\right)$, with $\xi\in\R^n$, is weak-Lipschitz with constant $w$,
\end{enumerate}
where $T(\cdot):\R^n\mapsto\R^n$ is a suitable diffeomorphism.
Then, the free synchronization problem for the multi-agent system is solvable with distributed $PI^hD^{n-1}$ controllers, with $h\geq 0$, of the form\footnote{With a slight abuse of notation with $h=0$ we mean here a pure proportional action.}
\begin{align*}
u^{(i)}(t)=&\frac{1}{\mathcal{L}_g\mathcal{L}_f^{n-1}(x^{(i)})}l\sum_{k=1}^n\sum_{j=1}^N l_{PD, kij}\left[T_k\left(x^{(j)}(t)\right)-\right. \\
&\left. T_k\left(x^{(i)}(t)\right)\right] + \\
& l\sum_{m=1}^h\sum_{j=1}^N l_{I, mij}\int_0^{t,(m)}\left[T_1\left(x^{(j)}(\tau)\right)-\right. \\
& \left. T_1\left(x^{(i)}(\tau)\right)\right] d \tau,
\end{align*} 
for all  $i=1,\dots,N$, and with the gains $l, l_{PD, kij}, l_{I, mij}$ selected according to Theorem \ref{thm:PI_controller_companion_form} and with $T_k(\cdot)$ being the $k$-th element of $T(\cdot)$.
\end{thm}
\begin{proof}
The proof of Theorem \ref{thm:synch_controllable_nonlinear_systems} is given in the Appendix.
\end{proof}

\begin{cor}\label{thm:synch_controllable_linear_systems}
Let us consider a connected graph $\mathcal{G}$ and a multi-agent system of linear dynamical agents of the form
$
\dot{x}^{(i)}=Ax^{(i)}+bu^{(i)}, 
$
with $x^{(i)}\in\R^n$ and $u^{(i)}\in \R$. If the pair $(A,b)$ is controllable, then there exists a full rank matrix $T$ such that the free synchronization problem for the multi-agent system is solvable with distributed $PI^hD^{n-1}$ controllers of the form
\begin{align*}
u^{(i)}(t)= & l\sum_{k=1}^n\sum_{j=1}^N l_{PD, kij}\left[T_k x^{(j)}(t)-T_k x^{(i)}(t)\right] +\\
& l\sum_{m=1}^h\sum_{j=1}^N l_{I, mij}\int_0^{t,(m)}\left[T_1 x^{(j)}(\tau) -T_1 x^{(i)}(\tau)\right] d \tau, 
\end{align*}
with the gains $l, l_{PD, kij}, l_{I, mij}$ selected according to Theorem \ref{thm:PI_controller_companion_form} and with $T_k$ being the $k$-th row of matrix $T$.
\end{cor}
\begin{proof}
The proof of Corollary \ref{thm:synch_controllable_linear_systems} is given in the Appendix.
\end{proof}

\begin{rem}
Notice that, from the above result the controllability hypothesis suffices to guarantee the synchronizability of the agents, as already showed in a different way in \cite{lidu:10}. However, it is worth noticing that the approach presented here naturally allows to explicitly consider integral control actions for possible disturbances rejections.  
\end{rem}

\section{Numerical example}\label{sec:numerical_examples}
In this section we show the effectiveness of our results on two numerical examples. Specifically, synchronization of nonlinear and linear oscillators with possible disturbances  will be achieved via the coupling selection illustrated in Section \ref{sec:ditributed_gain_selection}.

\subsection{Synchronization of Van der Pol oscillators}
We consider a network of ten identical Van der Pol oscillators whose model is given by the following relation 
\begin{eqnarray}\label{eq:van_der_pol_model}
\dot{x}_1^{(i)}&=&x_2^{(i)} \nonumber \\
\dot{x}_2^{(i)}&=&-x_1^{(i)}+\mu(1-|x_1^{(i)}|)x_2^{(i)}+u^{(i)}. \nonumber\\
\end{eqnarray}
For our example, we choose the parameter $\mu=2.5$ and initial conditions randomly assigned in the interval $[0,5]$ both for $x_1^{(i)}$ and $x_2^{(i)}$, for all the systems in the network. 
We validate Theorem \ref{thm:P_controller_companion_form} via creating a  connected random graph $\mathcal{G}$ which set the distributed control for the ten systems and a $(\mathcal{G},2)${\em -collection} over such graph (notice that, in this case $n=2$). 

Figure \ref{fig:van_der_pol_uncoupled_x1} shows the first state component of the networked systems when no coupling is considered, while the effect of the coupling of the assigned $(\mathcal{G},2)${\em -collection} allows the network to synchronize over a common manifold (Figure \ref{fig:van_der_pol_coupled_x1}). 
The synchronization error is depicted in Figure \ref{fig:van_der_pol_coupled_error}.

\begin{figure}[h!]
\centering
\subfigure[]{\label{fig:van_der_pol_uncoupled_x1}\includegraphics[width=0.4\textwidth]{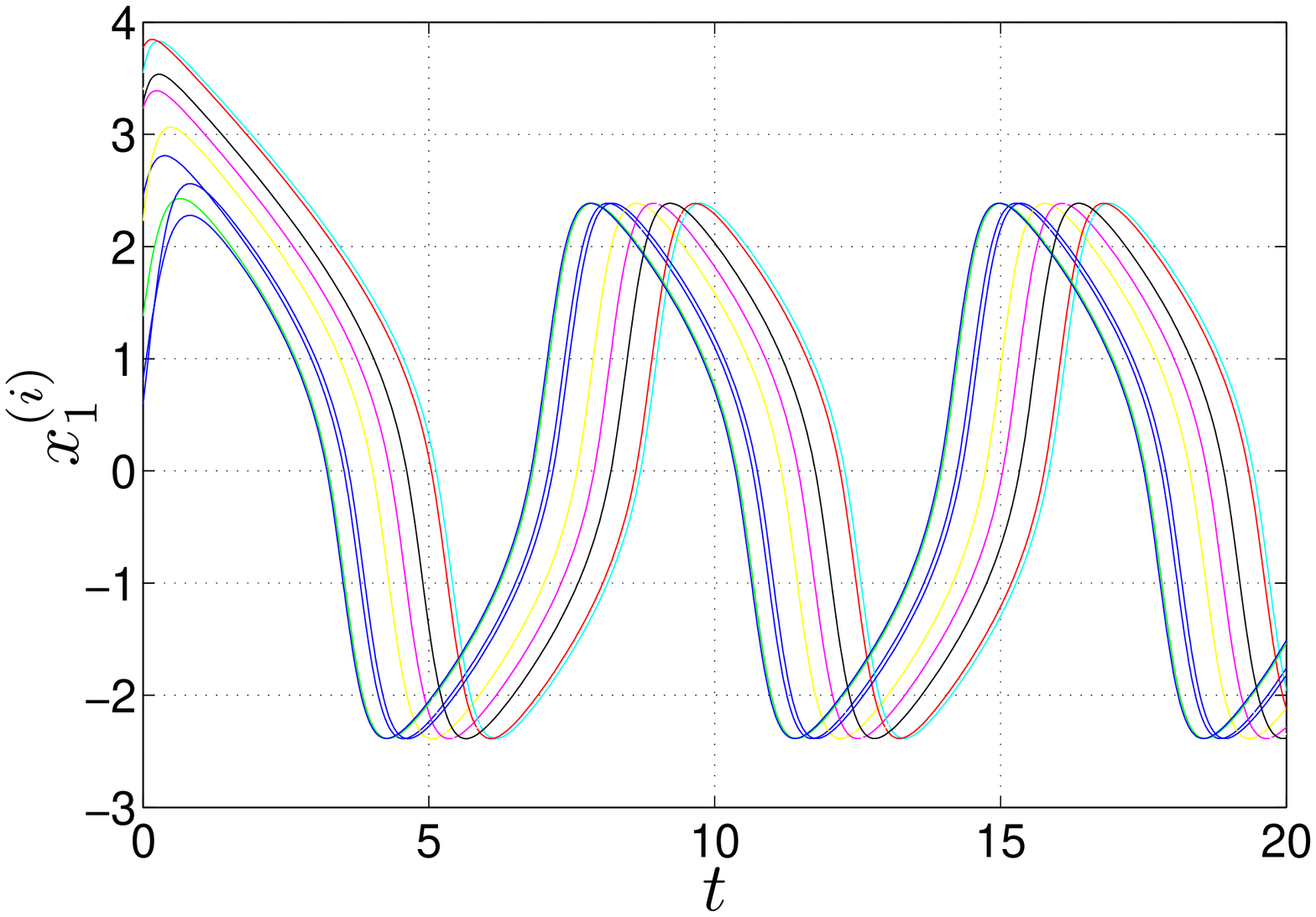}}
\subfigure[]{\label{fig:van_der_pol_coupled_x1}\includegraphics[width=0.4\textwidth]{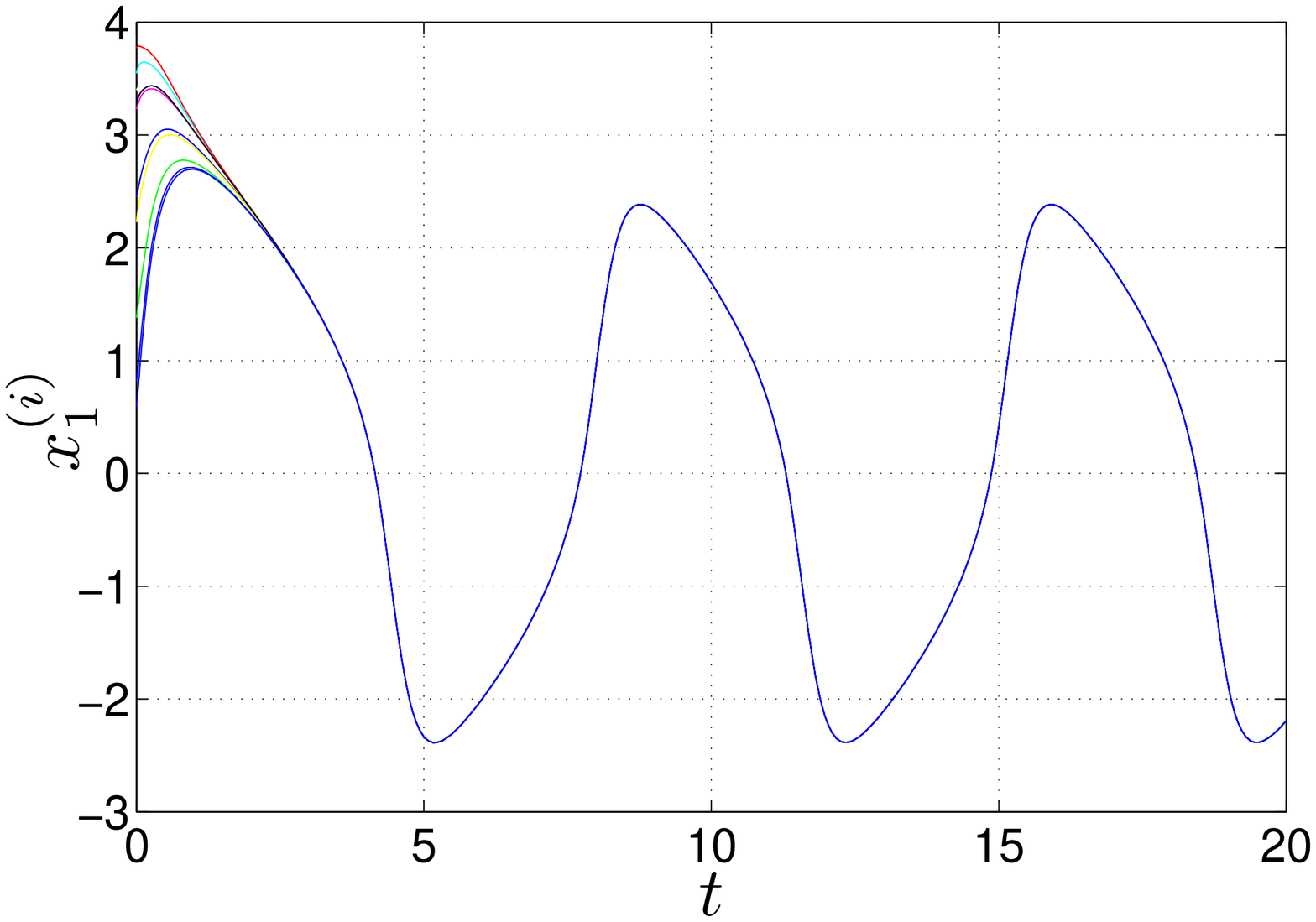}}
\caption{Time evolution of the state components $x_1^{(i)}$ for the network of Van der Pol oscillators: (a) uncoupled case; (b) coupled case.}
\label{fig:van_der_pol_x1}
\end{figure}

\begin{figure}[h!]
\centering
\includegraphics[width=0.4\textwidth]{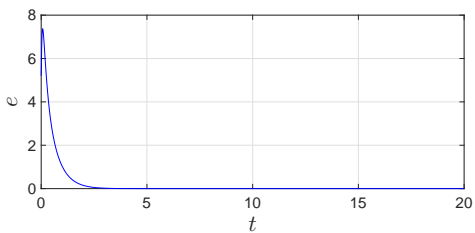}
\caption{Time evolution of the synchronization error $e$ for the network of Van der Pol oscillators.}
\label{fig:van_der_pol_coupled_error}
\end{figure}

\subsection{Synchronization of linear oscillators}
We now consider the synchronization of ten interconnected  linear oscillators
\begin{equation*}
\dot{x}^{(i)}=\left(\begin{array}{cc}
4  & 5\\
-5 & -4
\end{array}\right)x^{(i)}+
\left(\begin{array}{c}
1\\
1
\end{array}\right)u^{(i)}, 
\end{equation*}
using a $PD$ and a $PID$ controller according to Corollary \ref{thm:synch_controllable_linear_systems}.
Specifically, as done for the previous numerical example, we validate a distribute $PD$ controller via generating a connected random graph $\mathcal{G}$  for the overall system and a related $(\mathcal{G},2)${\em -collection}. 
It is easy to see that the system considered  is controllable and so, we use the distributed controller  given in Corollary \ref{thm:synch_controllable_linear_systems}, where the transformation matrix $T$ can {be shown} to be
\begin{equation*}
T=\left(\begin{array}{cc}
0.0556  & -0.0556\\
 0.5  & 0.5
\end{array}\right).
\end{equation*}
From Figure \ref{fig:linear_oscillator_nodisturbance} it is possible to see the time evolution of the first state component both for the case of uncoupled (Figure \ref{fig:linear_oscillator_nodisturbance_uncoupled_x1}) and coupled (Figure \ref{fig:linear_oscillator_nodisturbance_coupled_P_x1}) network, starting from randomly distributed initial conditions in the interval $[-10,10]$ for both the state components.

\begin{figure}[h!]
\centering
\subfigure[]{\label{fig:linear_oscillator_nodisturbance_uncoupled_x1}\includegraphics[width=0.4\textwidth]{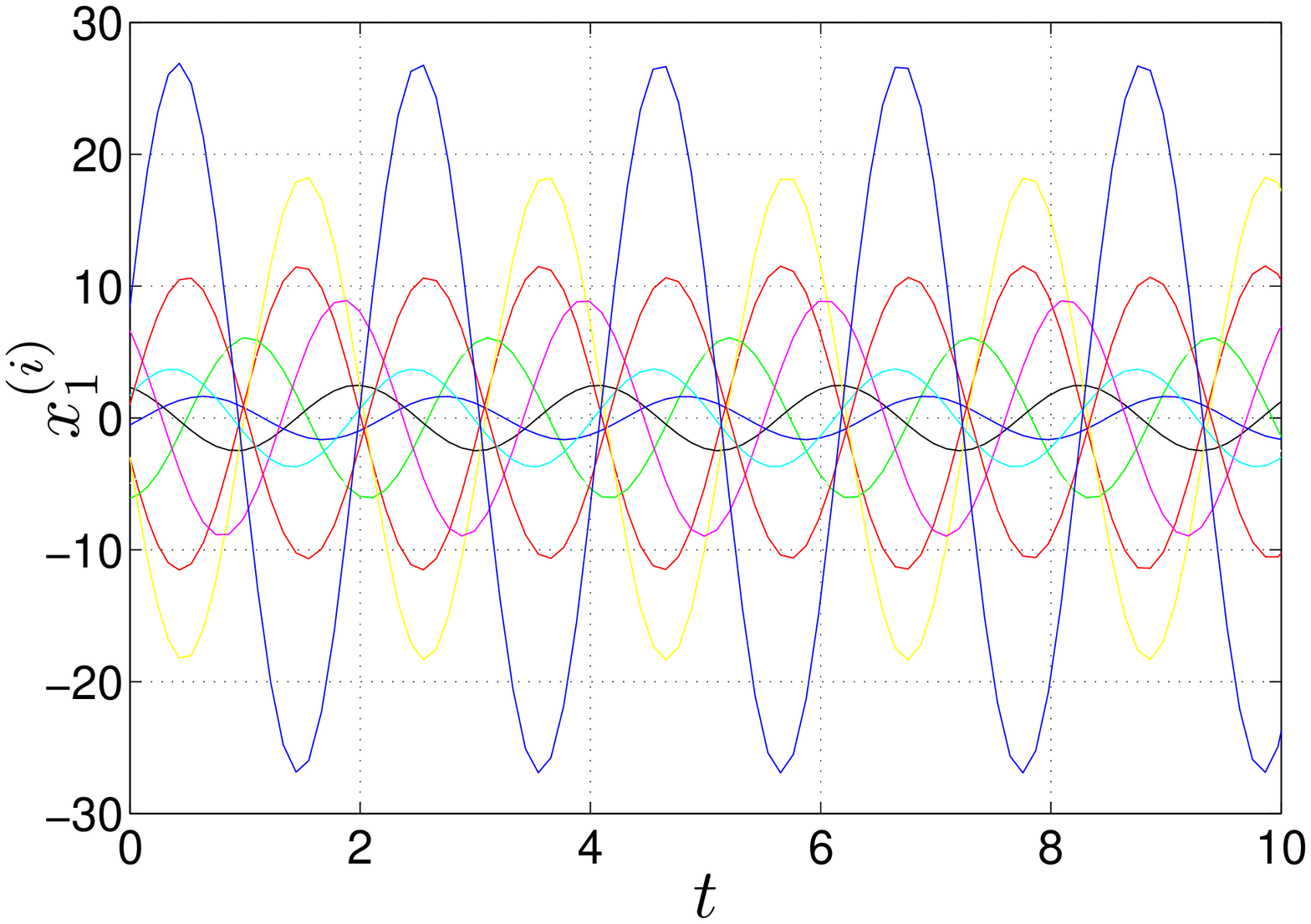}}
\subfigure[]{\label{fig:linear_oscillator_nodisturbance_coupled_P_x1}\includegraphics[width=0.4\textwidth]{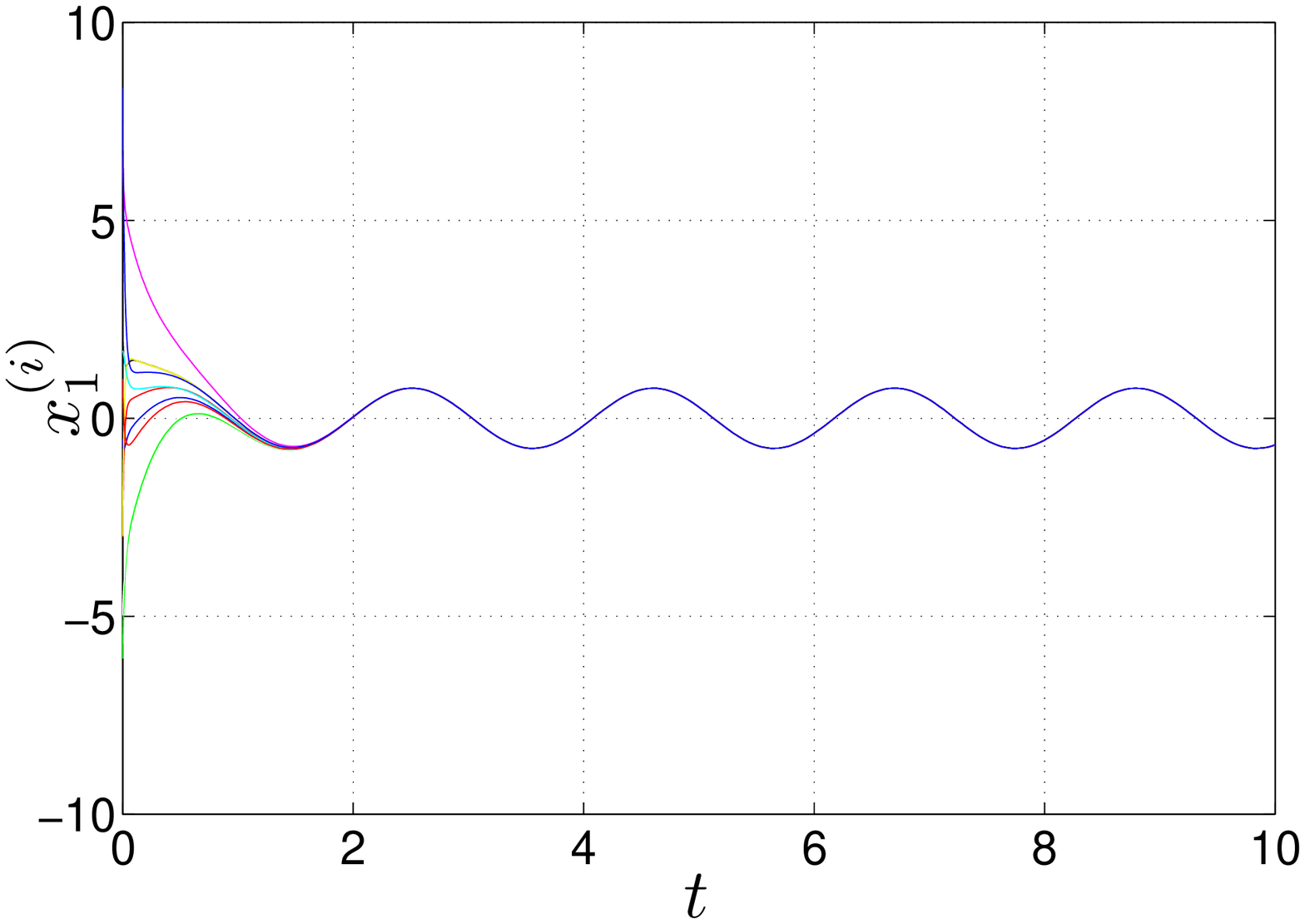}}
\caption{Time evolution of the state components $x_1^{(i)}$ for the network of linear oscillators with $PD$ controllers and no disturbances: (a) uncoupled case; (b) coupled case.}
\label{fig:linear_oscillator_nodisturbance}
\end{figure}

In order to validate the effectiveness of the distributed integral action, we add a step disturbance on a system in the network. In Figure \ref{fig:linear_oscillator_disturbance_coupled_P_x1} the first state component is again showed. It is possible to see that synchronization is no longer achieved. The residual global
synchronization error {reaches a constant value in the limit when $t\rightarrow +\infty$, which is equal to $e_\infty=5.8$} as depicted in  Figure \ref{fig:linear_oscillator_disturbance_coupled_P_error}.

\begin{figure}[h!]
\centering
\subfigure[]{\label{fig:linear_oscillator_disturbance_coupled_P_x1}\includegraphics[width=0.4\textwidth]{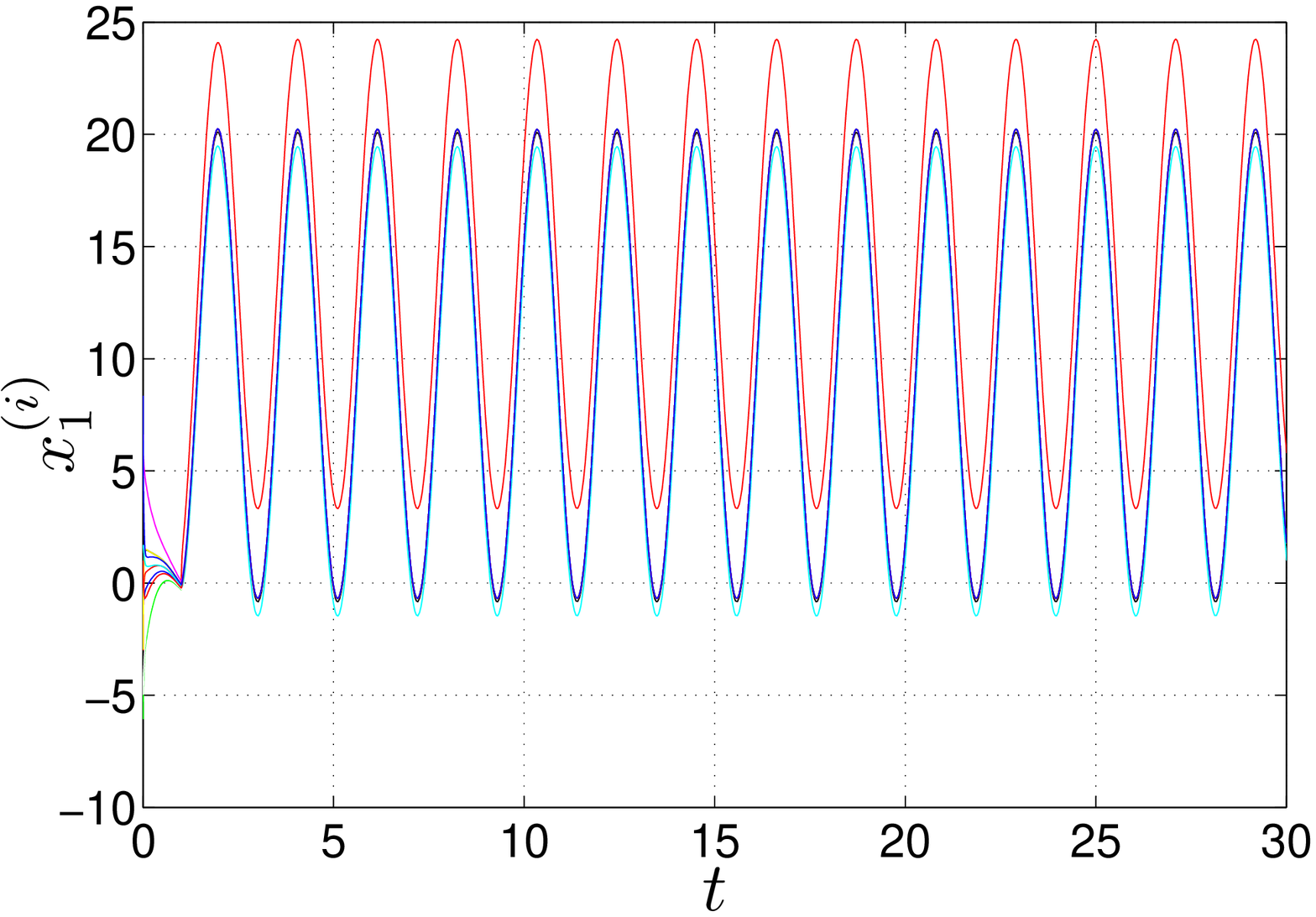}}
\subfigure[]{\label{fig:linear_oscillator_disturbance_coupled_P_error}\includegraphics[width=0.4\textwidth]{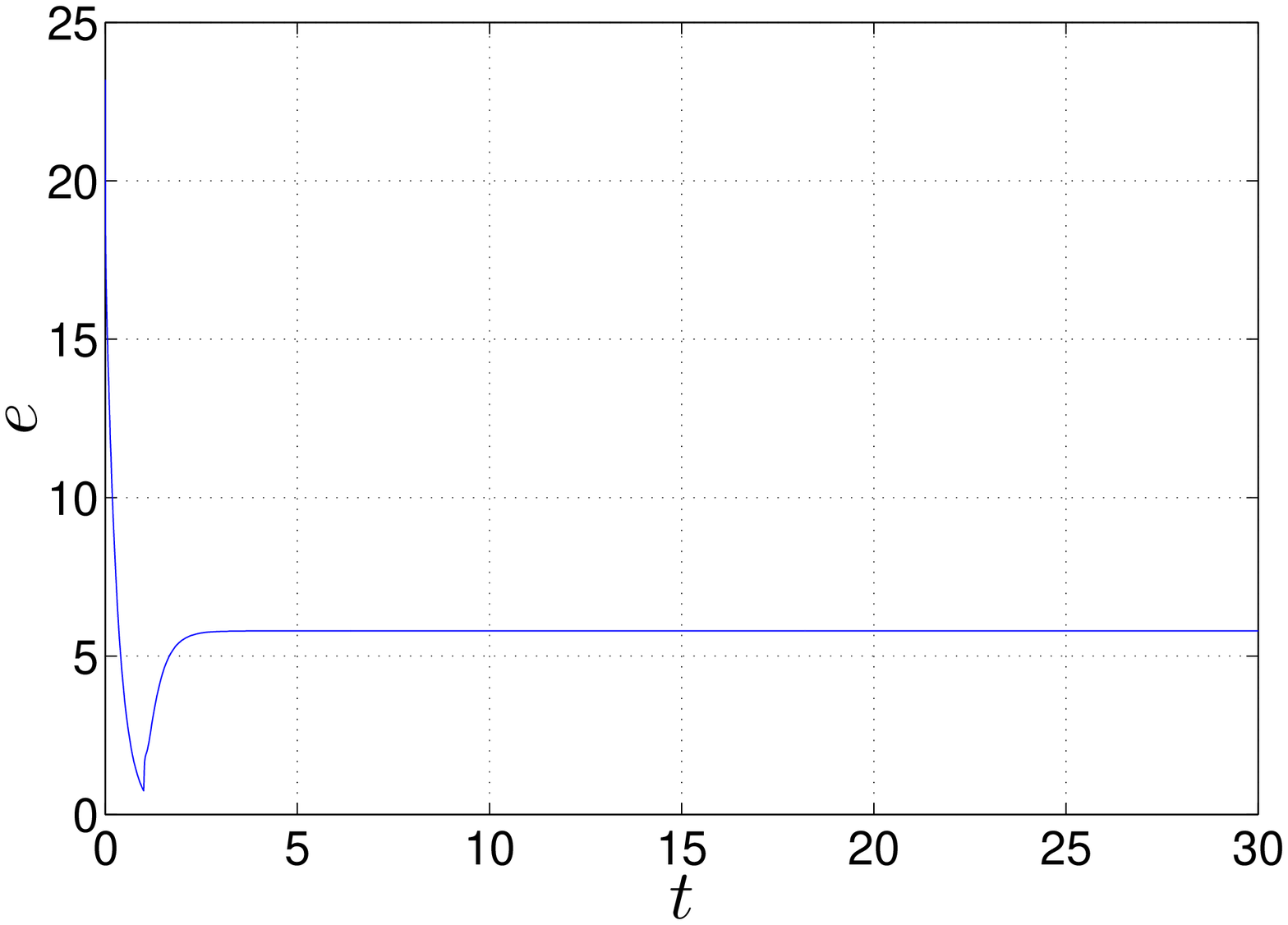}}
\caption{Time evolution of the network of linear oscillators with $PD$ controllers and heterogeneous  disturbances: (a) state components $x_1^{(i)}$; (b) global synchronization error $e$.}
\label{fig:linear_oscillator_disturbance_coupled_P}
\end{figure}

In order to reject the disturbance, we then consider a distributed $PID$ controller, coupling the network via a  $(\mathcal{G},3)${\em -collection} (notice that we have now $n=2$ and $h=1$).
As clearly emerges from Figure \ref{fig:linear_oscillator_disturbance_coupled_PI_x1_span80}, the integral control action is able to reject constant heterogeneous disturbances, thus leading the network to synchronization. 
{In Figures  \ref{fig:linear_oscillator_disturbance_coupled_PI_x1_span1_10}-\ref{fig:linear_oscillator_disturbance_coupled_PI_error_span70_80} the same evolution is given, zooming for a time span of ten seconds at the beginning and at the end of the simulation horizon, respectively. As can be witnessed, and differently from what happens in Figure \ref{fig:linear_oscillator_disturbance_coupled_P_x1}, all the nodes converge to the same oscillatory orbit.} 

Figure \ref{fig:linear_oscillator_disturbance_coupled_PI_error_span80} shows the asymptotic convergence to zero of the global synchronization error associated with such $PID$ scheme, {in comparison with the case in Figure \ref{fig:linear_oscillator_disturbance_coupled_P_error} where no integral action is considered.}

\begin{figure}[h!]
\centering
\subfigure[]{\label{fig:linear_oscillator_disturbance_coupled_PI_x1_span80}\includegraphics[width=0.4\textwidth]{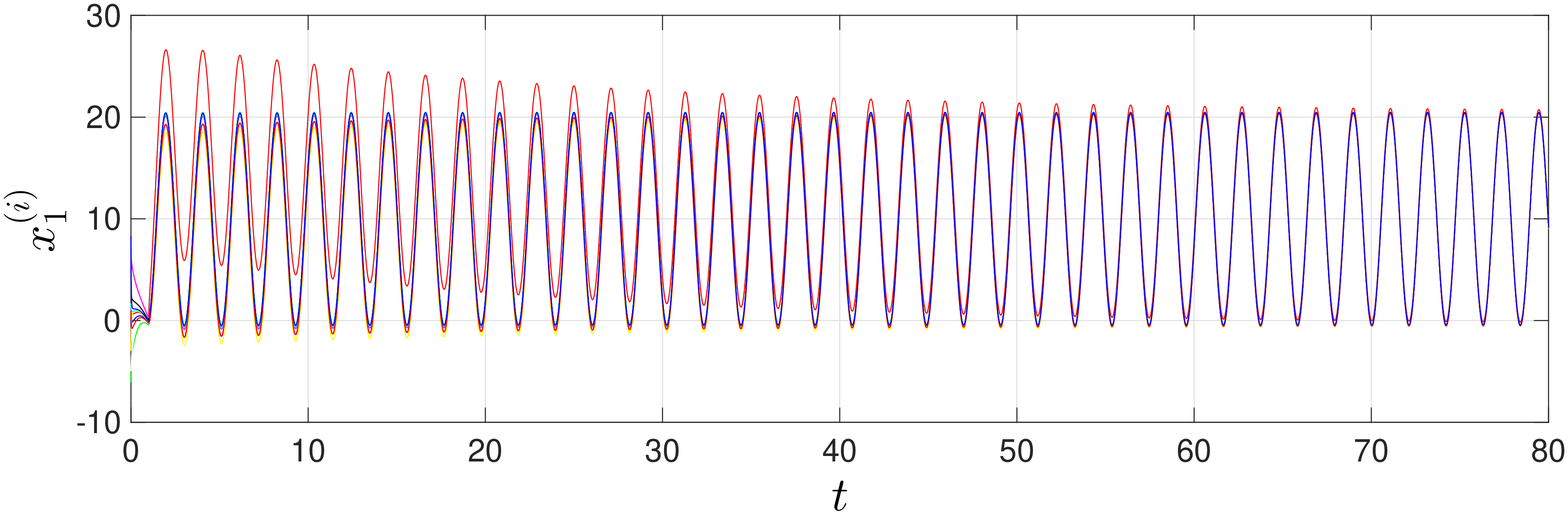}}
\subfigure[]{\label{fig:linear_oscillator_disturbance_coupled_PI_error_span80}\includegraphics[width=0.4\textwidth]{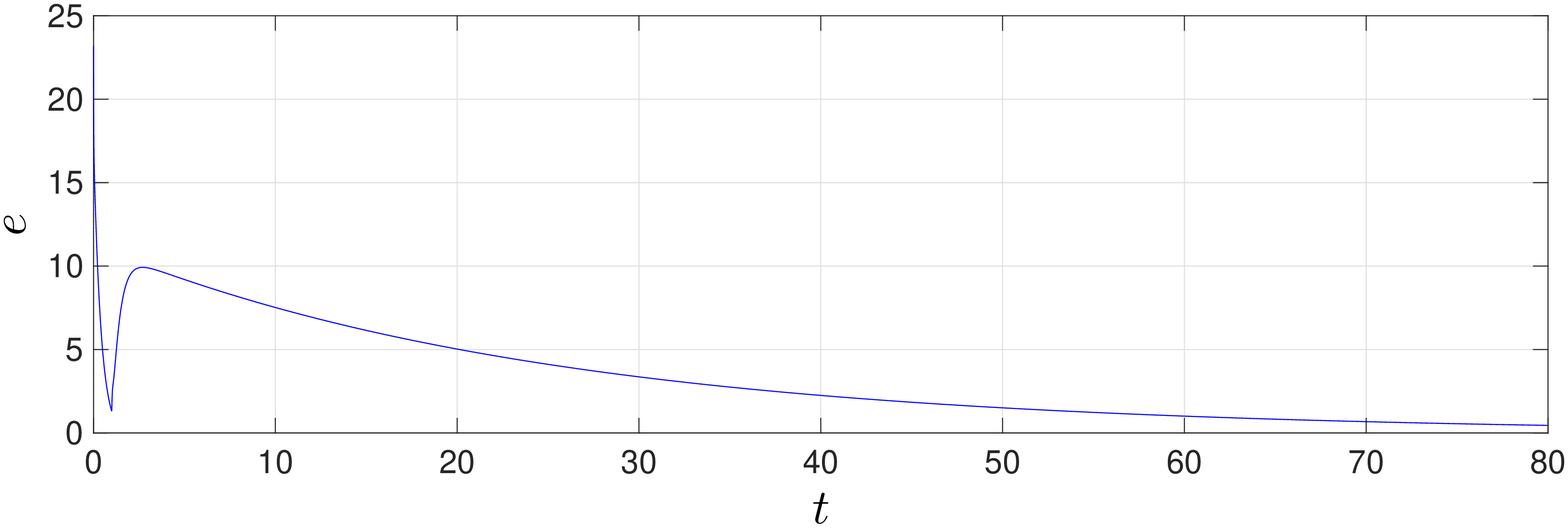}}
\caption{Time evolution of the network of linear oscillators with $PID$ controllers and heterogeneous  disturbances: (a) state components $x_1^{(i)}$; (b) global synchronization error $e$.}
\label{fig:linear_oscillator_disturbance_coupled_PI_span80}
\end{figure}

\begin{figure}[h!]
\centering
\subfigure[]{\label{fig:linear_oscillator_disturbance_coupled_PI_x1_span1_10}\includegraphics[width=0.4\textwidth]{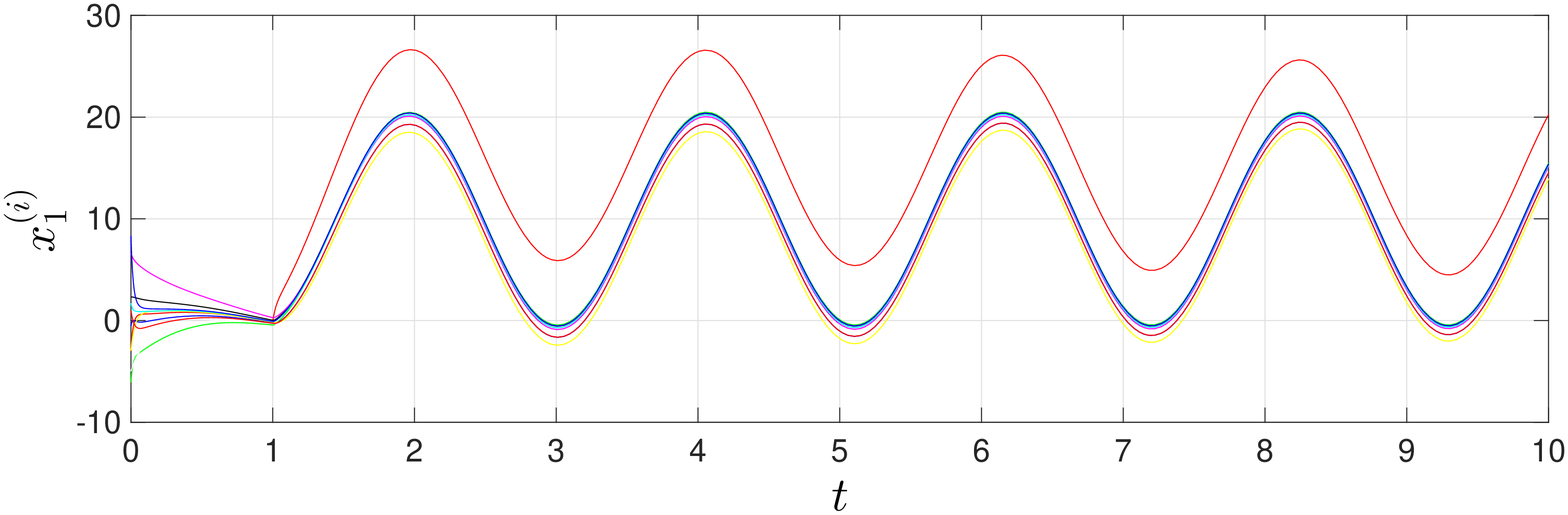}}
\subfigure[]{\label{fig:linear_oscillator_disturbance_coupled_PI_error_span70_80}\includegraphics[width=0.4\textwidth]{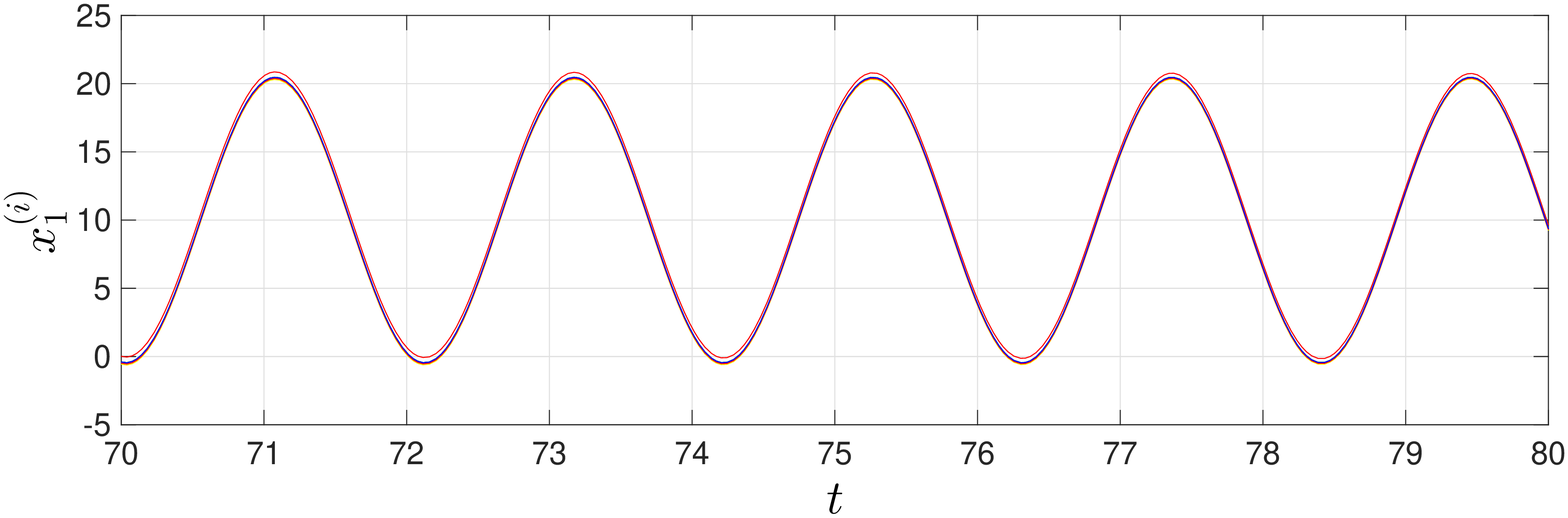}}
\caption{Time evolution of the network of linear oscillators with $PID$ controllers and heterogeneous  disturbances. Zoom of state components $x_1^{(i)}$: (a) beginning of the simulation horizon; (b) end of the simulation horizon.}
\label{fig:linear_oscillator_disturbance_coupled_PI_zoom}
\end{figure}

\section{Discussion and future work}\label{sec:conlusion_future_work}

In this paper we addressed the problem of higher-order free synchronization for nonlinear systems. 
Via an iterative procedure, we proved the existence of a class of feedback matrices, able to guarantee distributed state synchronization over any connected graph topology. The framework is related to any system order and easily embeds a possible distributed integral action of any order. 
The case of higher-order consensus is naturally embedded in our results as a particular case.
Furthermore, the methodology can also be extended to those linear and nonlinear systems admitting a (local) canonical transformation. In particular, for the specific case of linear systems, the synchronization with distributed $PI^hD^{n-1}$ controllers is guaranteed under the mild hypothesis of controllability of the agent's dynamics.

The presence of a distributed integral control action allows to attenuate possible distributed heterogeneous disturbances affecting the agents and, as shown in the numerical simulations, greatly improves the convergence performances. 

Future work will address in detail the analysis of robust synchronization of agents with parameters' mismatch and subjected to heterogeneous noises/disturbances as well as the case of directed/pinned network.

{A future direction of investigation is to recast the methodology adopted in this paper to the discrete time case. At the current stage, such extension is not trivial since the whole analysis (definitions of matrices $M_1$ and $H_1$, Algorithm \ref{alg:spectral_constranit_centralized_info} and Algorithm \ref{alg:spectral_constranit_decentralized_info}) is conducted for the continuous time case. Therefore, the discrete time case requires further studies.}


%
%
%
%
%
%
%
%
%
%
%
%
%
%
%
%
\bibliographystyle{IEEEtran}
\bibliography{IEEEabrv,bibliografiagenerale1_17}

\begin{thebibliography}{10}
\providecommand{\url}[1]{#1}
\csname url@samestyle\endcsname
\providecommand{\newblock}{\relax}
\providecommand{\bibinfo}[2]{#2}
\providecommand{\BIBentrySTDinterwordspacing}{\spaceskip=0pt\relax}
\providecommand{\BIBentryALTinterwordstretchfactor}{4}
\providecommand{\BIBentryALTinterwordspacing}{\spaceskip=\fontdimen2\font plus
\BIBentryALTinterwordstretchfactor\fontdimen3\font minus
  \fontdimen4\font\relax}
\providecommand{\BIBforeignlanguage}[2]{{%
\expandafter\ifx\csname l@#1\endcsname\relax
\typeout{** WARNING: IEEEtran.bst: No hyphenation pattern has been}%
\typeout{** loaded for the language `#1'. Using the pattern for}%
\typeout{** the default language instead.}%
\else
\language=\csname l@#1\endcsname
\fi
#2}}
\providecommand{\BIBdecl}{\relax}
\BIBdecl

\bibitem{ne:03}
M.~E.~J. Newman, ``The structure and function of complex networks,'' \emph{Siam
  review}, vol.~45, no.~2, pp. 167--256, 2003.

\bibitem{bola:06}
S.~Boccaletti, V.~Latora, Y.~Moreno, M.~Chavez, and D.~U. Hwang, ``Complex
  networks: structure and dynamics.'' \emph{Physics Reports}, vol. 424, pp.
  175--308, 2006.

\bibitem{bape:02}
M.~Barahona and L.~M. Pecora, ``Synchronization in small-world systems,''
  \emph{Physical Review Letters}, vol.~89, no.~5, p. 054101, 2002.

\bibitem{ar:07}
M.~Arcak, ``Passivity as a design tool for group coordination,'' \emph{IEEE
  Trans. on Automatic Control}, vol.~52, no.~8, pp. 1380--1390, 2007.

\bibitem{caan:08}
M.~Cao, B.~D.~O. Anderson, A.~S. Morse, and C.~Yu, ``Control of acyclic
  formations of mobile autonomous agents,'' in \emph{47th IEEE Conference on
  Decision and Control}, 2008.

\bibitem{ol:07}
R.~Olfati-Saber, ``Distributed kalman filtering for sensor networks,'' in
  \emph{46th IEEE Conference on Decision and Control}, 2007.

\bibitem{hast:12}
A.~Hamadeh, G.-B. Stan, R.~Sepulchre, and J.~Gon\c{c}alves, ``Global state
  synchronization in networks of cyclic feedback systems,'' \emph{IEEE Trans.
  on Automatic Control}, vol.~57, no.~2, pp. 478--483, 2012.

\bibitem{hich:06}
D.~J. Hill and G.~Chen, ``Power systems as dynamic networks,'' in
  \emph{Proceedings of the IEEE International Symposium on Circuits and
  Systems}, 2006.

\bibitem{dobu:12}
F.~D\"orfler and F.~Bullo, ``Synchronization and transient stability in power
  networks and nonuniform kuramoto oscillators,'' \emph{SIAM Journal on Control
  and Optimization}, vol.~50, no.~3, pp. 1616--1642, 2012.

\bibitem{scse:09}
L.~Scardovi and R.~Sepulchre, ``Synchronization in networks of identical linear
  systems,'' \emph{Automatica}, vol.~45, no.~11, pp. 2557--2562, 2009.

\bibitem{sesh:09}
J.~H. Seo, H.~Shim, and J.~Back, ``Consenus of high-order linear systems using
  dynamic output feedback compensator: low gain approach,'' \emph{Automatica},
  vol.~45, pp. 2659--2664, 2009.

\bibitem{lidu:10}
Z.~Li, Z.~Duan, G.~Chen, and L.~Huang, ``Consensus of mulit-agent systems and
  synchronization of complex networks: a unified viewpoint,'' \emph{IEEE Trans.
  Circuits Syst. I}, vol.~57, no.~5, pp. 213--224, 2010.

\bibitem{zhle:11}
H.~Zhang, F.~L. Lewis, and A.~Das, ``Optimal design for synchronization of
  cooperative systems: State feedback, observer and output feedback,''
  \emph{IEEE Trans. on Automatic Control}, vol.~56, no.~8, pp. 1948--1952,
  2011.

\bibitem{lihizh:12}
T.~Liu, D.~J. Hill, and J.~Zhao, ``Synchronization of dynamical networks by
  network control,'' \emph{IEEE Trans. on Automatic Control}, vol.~57, no.~6,
  pp. 1574--1580, 2012.

\bibitem{lich:06}
Z.~Li and G.~Chen, ``Global synchronization and asymptotic stability of complex
  dynamical networks,'' \emph{IEEE. Trans. Circuits Syst. II}, vol.~53, pp.
  28--33, 2006.

\bibitem{wasl:05}
W.~Wang and J.~J.~E. Slotine, ``On partial contraction analysis for coupled
  nonlinear oscillators,'' \emph{Biol. Cybern.}, vol.~92, no.~1, pp. 38--53,
  2005.

\bibitem{dili:14}
M.~di~Bernardo, D.~Liuzza, and G.~Russo, ``Contraction analysis for a class of
  nondifferentiable systems with applications to stability and network
  synchronization,'' \emph{SIAM J. Control Optim.}, vol.~52, no.~5, pp.
  3203--3227, 2014.

\bibitem{scar:10}
L.~Scardovi, M.~Arcak, and E.~D. Sontag, ``Synchronization of interconnected
  systems with applications to biochemical networks: An input-output
  approach,'' \emph{IEEE Trans. on Automatic Control}, vol.~55, no.~6, pp.
  1367--1379, 2010.

\bibitem{lihi:11}
T.~Liu, J.~Hill, and J.~Zhao, ``Incremental-dissipativity-based synchronization
  of interconnected systems,'' in \emph{Proceedings of the 18th IFAC World
  Congress}, 2011.

\bibitem{slli:91}
J.~J. Slotine and W.~Li, \emph{Applied nonlinear control}.\hskip 1em plus 0.5em
  minus 0.4em\relax Prentice Hall (Englewood Cliffs, NJ, USA), 1991.

\bibitem{liji:08}
P.~Lin, Y.~Jia, and L.~Li, ``Distributed robust {$H_\infty$} consensus control
  in directed networks of agents with time-delay,'' \emph{Systems \& Control
  Letters}, vol.~57, pp. 643--653, 2008.

\bibitem{lili:11}
Z.~Li, X.~Liu, and M.~Fu, ``Global consensus control of {L}ipschitz nonlinear
  mulit-agent systems,'' in \emph{Proceedings of the 18th IFAC World Congress},
  2011.

\bibitem{yuchca:11}
W.~Yu, G.~Chen, and M.~Cao, ``Consensus in directed networks of agents with
  nonlinear dynamics,'' \emph{IEEE Transactions on Automatic Control}, vol.~56,
  no.~6, pp. 1436--1441, 2011.

\bibitem{wehu:13}
G.~Wen, G.~Hu, W.~Yu, J.~Cao, and G.~Chen, ``Consensus tracking for
  higher-order multi-agent systems with switching directed topologies and
  occasionally missing control inputs,'' \emph{Systems \& Control}, vol.~62,
  pp. 1151--1158, 2013.

\bibitem{lich:15}
X.~Liu and T.~Chen, ``Synchronization of complex networks via aperiodically
  intermittent pinning control,'' \emph{IEEE Transaction on Automatic Control},
  vol.~60, no.~12, pp. 3316--3321, 2015.

\bibitem{dedi:15}
P.~DeLellis, M.~diBernardo, and D.~Liuzza, ``Convergence and synchronization in
  heterogeneous networks of smooth and piecewise smooth systems,''
  \emph{Automatica}, vol.~56, pp. 1--11, 2015.

\bibitem{reat:05}
W.~Ren and E.~Atkins, ``Second-order consensus protocols in multiple vehicle
  systems with local interactions,'' in \emph{AIAA Guidance, Navigation, and
  Control Conference and Exhibit}, 2005.

\bibitem{yuch_consensus:10}
W.~Yu, G.~Chen, and M.~Cao, ``Some necessary and sufficient conditions for
  second-order consensus in multi-agent dynamical systems,'' \emph{Automatica},
  vol.~46, no.~6, pp. 1089--1095, 2010.

\bibitem{yuzh:13}
W.~Yu, L.~Zhou, X.~Yu, J.~L\"u, and R.~Lu, ``Consensus in multi-agent systems
  with second-order dynamics and sampled data,'' \emph{IEEE Transactions on
  Industrial Informatics}, vol.~9, no.~4, pp. 2137--2146, 2013.

\bibitem{remo:06}
W.~Ren, K.~Moore, and Y.~Q. Chen, ``High-order consensus algorithms in
  cooperative vehicle systems,'' in \emph{Networking, Sensing and Control,
  2006. ICNSC'06. Proceedings of the 2006 IEEE International Conference on},
  2006.

\bibitem{hoga:07}
Y.~Ho, L.~Gao, D.~Cheng, and J.~Hu, ``Lyapunov-based approach to multiagent
  systems with switching jointly connected interconnection,'' \emph{IEEE
  Transactions on Automatic Control}, vol.~52, no.~5, pp. 943--948, 2007.

\bibitem{yuchcaku:10}
W.~Yu, G.~Chen, M.~Cao, and J.~Kurths, ``Second-order consensus for multiagent
  systems with directed topologies and nonlinear dynamics,'' \emph{IEEE
  Transactions on Systems, Man and Cybernetics \-- Part B: Cybernetics},
  vol.~40, no.~3, pp. 881--891, 2010.

\bibitem{soca:10}
Q.~Song, J.~Cao, and W.~Yu, ``Second-order leader-following consensus of
  nonlinear multi-agent systems via pinning control,'' \emph{Systems \& Control
  Letters}, vol.~59, pp. 553--562, 2010.

\bibitem{re:08}
W.~Ren, ``On consensus algorithms for double-integrator dynamics,'' \emph{IEEE
  Trans. on Automatic Control}, vol.~58, no.~6, pp. 1503--1509, 2008.

\bibitem{yuch:11}
W.~Yu, G.~Chen, W.~Ren, J.~Kurths, and W.~X. Zheng, ``Distributed higher order
  consensus protocols in multiagent dynamical systems,'' \emph{IEEE Trans.
  Circuits Syst. I}, vol.~58, no.~8, 2011.

\bibitem{lixi:13}
K.~Liu, G.~Xie, W.~Ren, and L.~Wang, ``Consenus for mulit-agent systems with
  inherent nonlinear dynamics under directed topologies,'' \emph{Systems \&
  Control Letters}, vol.~62, no.~2, pp. 152--162, 2013.

\bibitem{dale:11}
A.~Das and F.~L. Lewis, ``Cooperative adaptive control for synchronization of
  second-order systems with unknown nonlinearities,'' \emph{International
  Journal of Robust and Nonlinear Control}, vol.~21, pp. 1509--1524, 2011.

\bibitem{zhle:10}
H.~Zhang and L.~Lewis, ``Synchronization of networked higher-order nonlinear
  systems with unknown dynamics,'' in \emph{49th IEEE Conference on Decision
  and Control}, 2010.

\bibitem{zhle:12}
------, ``Adaptive cooperative tracking control of higher-order nonlinear
  systems with unknown dynamics,'' \emph{Automatica}, vol.~48, no.~7, pp.
  1432--1439, 2012.

\bibitem{bile:14}
A.~Bidram, F.~Lewis, and A.~Davoudi, ``Distributed control systems for
  small-scale power networks: Using multiagent cooperative control theory,''
  \emph{Control Systems, IEEE}, vol.~34, no.~6, pp. 56--77, 2014.

\bibitem{frya:06}
R.~A. Freeman, P.~Yang, and K.~M. Lynch, ``Stability and convergence properties
  of dynamic average consensus estimators,'' in \emph{45th Conference on
  Decision and Control}, 2006.

\bibitem{chwa:16}
L.~Cheng, Y.~Wang, W.~Ren, Z.-G. Hou, and M.~Tan, ``Containment control of
  multiagent systems with dynamic leaders based on a {$PI^n$}-type approach,''
  \emph{IEEE Transactions on Cybernetics}, vol.~46, no.~12, pp. 3004--3017,
  2016.

\bibitem{wach:15}
Y.~Wang, L.~Cheng, Z.-G. Hou, M.~Tan, and H.~Yu, ``Coordinated transportation
  of a group of unmanned ground vehicles,'' in \emph{34th Chinese Control
  Conference}, 2015.

\bibitem{hojo:87}
R.~A. Horn and C.~R. Johnson, \emph{Matrix Analisis}, P.~S. of~the
  University~of Cambridge, Ed.\hskip 1em plus 0.5em minus 0.4em\relax Cambridge
  University Press, 1987.

\bibitem{kh:02}
H.~K. Khalil, \emph{Nonlinear Systems}.\hskip 1em plus 0.5em minus 0.4em\relax
  Prentice Hall, New Jersey, 2002.

\bibitem{isma:13}
A.~Isidori, L.~Marconi, and G.~Casadei, ``Robust output synchronisation of
  network of heterogeneous nonlinear agents via nonlinear regulation theory,''
  in \emph{52nd IEEE Conference on Decision and Control}, 2013.

\bibitem{goro:01}
C.~D. Godsil and G.~Royle, \emph{Algebraic Graph Theory}.\hskip 1em plus 0.5em
  minus 0.4em\relax Springer, 2001.

\bibitem{olfa:07}
R.~Olfati-Saber, A.~A. Fax, and R.~M. Murray, ``Consensus and cooperation in
  networked multi-agent systems,'' \emph{Proceedings of the IEEE}, vol.~95,
  no.~1, pp. 215--233, 2007.

\bibitem{lu:79}
D.~G. Luenberger, \emph{Introduction to Dynamic Systems: Theory, Models, and
  Applications}.\hskip 1em plus 0.5em minus 0.4em\relax Wiley, 1979.

\end{thebibliography}
\newpage
\appendix
\section*{Proof of Theorem \ref{thm:PI_controller_companion_form}}
Considering $x_k:=\left[x_k^{(1)},\dots,x_k^{(N)}\right]^T\in\R^N$, for all $k\in\{1,\dots,n\}$, as the stack of the $k$-th component of each agent,
we define the $(n+h)N$ stack system that embeds the integral control action considering the state position
 \begin{align*}
z_\vartheta(t) &= \int_{0}^{t,(h-\vartheta+1)}x_1 (\tau) d\tau, & \vartheta=1,\dots, h \\
z_\vartheta(t) &= x_{\vartheta-h}(t),  & \vartheta=h+1,\dots, h+n. 
\end{align*}
So, taking into account \eqref{eq:L_I_ridefine}-\eqref{eq:L_P_ridefine}, the following system system can be written as
\begin{align*}
&\dot{z}_1=z_2  \\
&\,\, \vdots   \\
&\dot{z}_h=z_{h+1}  \\
& \,\, \vdots  \\
&\dot{z}_{h+n}=F(t,z_h,\dots, z_{h+n})-l\sum_{\vartheta=1}^{n+h}L_{\vartheta}z_{\vartheta}, \\
\end{align*}
where we have used the stack form with $z:=\left[z_1^T,\dots,z_{n+h}^T\right]^T\in\R^{(n+h)N}$, and $F(t,z_h,\dots, z_{h+n}):=\left[f(t,z^{(1)}),\dots,f(t,z^{(N)})\right]^T$. 

The above system is again in companion form and Theorem \ref{thm:P_controller_companion_form} can be invoked to complete the proof.

\section*{Proof of Theorem \ref{thm:synch_controllable_nonlinear_systems}}
Conditions (i) and (ii) are necessary and sufficient for the existence of a diffeomorphism $T(\cdot)$ and for guaranteeing that $\mathcal{L}_g\mathcal{L}_f^{n-1}(x^{(i)})\neq 0$ \cite{slli:91}. Considering the state transformation $z^{(i)}=T(x^{(i)})$, system \eqref{eq:nonlinear_agent} can be written in the companion form
\begin{eqnarray}\label{eq:transformed_agent}
\dot{z}_1^{(i)}&=&z_2^{(i)} \nonumber \\
\dot{z}_2^{(i)}&=&z_3^{(i)} \nonumber \\
&\vdots & \nonumber \\
\dot{z}_n^{(i)}&=&\mathcal{L}_f^n\left(T^{-1}\left(z^{(i)}\right)\right)+\tilde{u}^{(i)}, \quad i=1,\dots, N, \nonumber\\
\end{eqnarray}
with $z^{(i)}=[z_1^{(i)}, \dots,z_N^{(i)}]^T$.
Now, since condition (iii) holds,  Theorem \ref{thm:PI_controller_companion_form} can be applied and $\tilde{u}_i(t)$ can be selected for the systems  \eqref{eq:transformed_agent} thus leading to $\lim_{t\rightarrow \infty}\|z^{(i)}(t)-z^{(j)}(t)\| =0$, for all $i,j=1,\dots,N$, or equivalently $\lim_{t\rightarrow \infty} z^{(i)}(t)=\lim_{t\rightarrow \infty} z^{(j)}(t)$. Now, since $T(\cdot)$ is a smooth invertible function, we have $\lim_{t\rightarrow \infty} T^{-1}\left(z^{(i)}(t)\right)=\lim_{t\rightarrow \infty} T^{-1}\left(z^{(j)}(t)\right)$ and so the convergence of the original states $x_i(t),x_j(t)$ is obtained.

\section*{Proof of Corollary \ref{thm:synch_controllable_linear_systems}}
Condition (i) of Theorem \ref{thm:synch_controllable_nonlinear_systems} is guaranteed by the controllability hypothesis, while condition (ii) is trivially satisfied by any set of constant vectors. So, the existence of a diffeomorphism is guaranteed and, furthermore, for controllable linear systems a linear state transformation $z^{(i)}=Tx^{(i)}$ can be explicitly given in closed form \cite{lu:79}. So, we can derive the equivalent linear system in companion form
\begin{eqnarray}
\dot{z}_1^{(i)}&=&z_2^{(i)} \nonumber \\
&\vdots & \nonumber \\
\dot{z}_n^{(i)}&=&a^Tz^{(i)}+u^{(i)}, \quad i=1,\dots, N, \nonumber\\
\end{eqnarray}
with $z^{(i)}=[z_1^{(i)}, \dots,z_n^{(i)}]^T$ and $a\in\R^n$ being a constant vector.
Now, since $a^Tz^{(i)}$ is a linear function, it is also Lipschitz and, by Lemma \ref{lem:weak_Lipschitz}, weak-Lipschitz. For this reason, also condition (iii) is  satisfied and the corollary is proved by invoking Theorem \ref{thm:synch_controllable_nonlinear_systems}. Furthermore, the coupling gains can be selected analogously to what done in Theorem \ref{thm:PI_controller_companion_form}.

\vfill
%
%
\end{document}